\documentclass[11pt]{amsart}

\usepackage[T1]{fontenc}
\usepackage[utf8]{inputenc}
\usepackage{lmodern}
\usepackage{amsmath,amssymb,amsthm,mathtools}
\usepackage{microtype}
\usepackage[margin=1.1in]{geometry}
\usepackage{enumitem}
\usepackage[hidelinks]{hyperref}
\usepackage{cleveref}

\theoremstyle{plain}
\newtheorem{theorem}{Theorem}[section]
\newtheorem{proposition}[theorem]{Proposition}
\newtheorem{lemma}[theorem]{Lemma}
\newtheorem{corollary}[theorem]{Corollary}
\theoremstyle{definition}
\newtheorem{definition}[theorem]{Definition}
\newtheorem{remark}[theorem]{Remark}
\newtheorem{example}[theorem]{Example}

\usepackage[dvipsnames,svgnames,table]{xcolor}
\definecolor{NiceBlue}{rgb}{0.2, 0.25, 0.8}
\newcommand{\struc}[1]{{\relax\color{NiceBlue}{\textit{#1}}}}

\newcommand{\M}{\mathcal{M}}

\newcommand{\im}{\operatorname{im}}
\newcommand{\rank}{\operatorname{rank}}
\newcommand{\Hess}{\operatorname{Hess}}
\newcommand{\Log}{\operatorname{Log}}
\newcommand{\Exp}{\operatorname{Exp}}

\newcommand{\dd}{\mathrm d}

\newcommand{\ip}[2]{\left\langle #1,#2\right\rangle}
\newcommand{\norm}[1]{\left\lVert #1\right\rVert}
\newcommand{\Isom}{\textsc{Isom}}
\newcommand{\R}{\mathbb{R}}
\newcommand{\Sph}{\mathbb{S}}
\newcommand{\cK}{\mathcal{K}}
\newcommand{\cO}{\mathcal{O}}

\title[Intrinsic Rigidity on Riemannian Manifolds]{Intrinsic Rigidity of Bar-and-Joint Frameworks on Riemannian Manifolds}
\author[Matthias Adrian-Himmelmann]{Matthias Adrian-Himmelmann$^\dagger$}
\thanks{$^\dagger$Institut für Analysis und Algebra, Universität Braunschweig, Germany.\\ Email: \href{mailto:matthias.himmelmann@tu-braunschweig.de}{matthias.himmelmann@tu-braunschweig.de}}
\date{}

\begin{document}

\begin{abstract}
We develop an intrinsic rigidity theory for bar-and-joint frameworks whose vertices are constrained to a Riemannian manifold and whose edges are given by minimizing geodesic arcs. Such frameworks are called rigid if they cannot be deformed beyond the ambient isometries, while preserving the geodesic edge lengths. For admissible placements, for which adjacent vertices are joined by unique minimizing geodesics, we define a coordinate-free rigidity operator as the differential of the squared geodesic distance of the edges. By the first-variation formula, its edge rows are expressed in terms of the Riemannian logarithm, providing an intrinsic analogue of the classical Euclidean rigidity matrix. Trivial motions are characterized by orbits of the diagonal action of the isometry group, leading to notions of infinitesimal and local rigidity. We further develop a covariant second-order rigidity theory using the Riemannian Hessians of the edge constraints and characterize second-order rigidity through self-stresses and associated quadratic forms. Finally, we show that infinitesimal and second-order rigidity can be used as sufficient conditions for the local rigidity of an intrinsic framework. Our overarching goal is to provide a unified differential-geometric formulation of intrinsic framework rigidity on smooth manifolds.
\end{abstract}
\maketitle



\section{Introduction}

Classical bar-and-joint rigidity studies a finite graph $G=(V,E)$ together with a placement of its vertices in Euclidean space.  Each edge represents a fixed-distance constraint. A framework is called rigid if these distance constraints do not permit deformations beyond the ambient isometries. This problem has many applications, ranging from soft-matter physics, across biochemistry to structural engineering. 
However, determining the rigidity of a bar-and-joint framework is computationally challenging \cite{abbott}, which in the past has promoted the development of efficiently verifiable sufficient conditions.

Beyond Euclidean space, the rigidity of bar-and-joint frameworks is studied in various other settings. For instance, frameworks whose vertices are constrained to surfaces in Euclidean space, while the bars retain ambient Euclidean lengths, are treated in \cite{NixonOwenPower2012}. First-order rigidity concepts in special Euclidean, spherical, and hyperbolic geometries are understood \cite{SaliolaWhiteley2007}. In \cite{dewarnormed,dewarnormed2}, the local, infinitesimal and even second-order rigidity theory are developed in normed spaces over the real numbers using Lie groups. There also exists a formulation of continuous and infinitesimal rigidity in metric and multiple-distance settings which includes several normed spaces \cite{cylindricalnormspaces,normedspaces} and geodesic distance constraints on specific surfaces \cite{NixonPower2020}. Moreover, with the help of hyperbolic symmetry groups, the rigidity on compact Riemann surfaces has been studied \cite{dewar2026rigiditycompactsurfaceshyperbolic}. In a more general setting, the rigidity for special types of polynomial maps is investigated. This relates to geometric constraint systems \cite{geometric_constraint_system} and $g$-rigidity \cite{g-rigidity}. 

With this article we strive to create a general theory which unifies these cases. By placing the framework's vertices on a Riemannian manifold and assuming that the edges are given by minimizing geodesic arcs, we are able to study framework rigidity in an intrinsic and natural setting. We formulate the intrinsic Riemannian differential-geometric theory in a coordinate-free form, which is valid on an arbitrary Riemannian manifold. 

This setting is far from artificial. Not only can we understand and generalize the rigidity of frameworks in contexts that have been previously described, but this approach also allows us to study the broader class of geometric constraint systems whose vertices are constrained to a smooth manifold and whose edges come with a smoothly varying measurement. This is, for instance, the case for 3-polytopes with edge length and facet planarity constraints, where the polytopal realization space forms a smooth manifold and the edges are equipped with a Euclidean distance constraint \cite{polytoperigidity1,polytoperigidity2}. 

In order to define rigidity for frameworks on Riemannian manifolds, the trivial deformations need to be factored out. For intrinsic bar-and-joint frameworks, the trivial motions are given by orbits of the connected component of the isometry Lie group that contains the identity. If beyond these orbits, there are no other sufficiently close realizations of the framework, then we call the framework rigid and otherwise flexible. The distinction between local and global isometries is important on general Riemannian manifolds, since local isometries need not extend to global ones. Consequently, deformations that locally preserve all pairwise distances may still be considered nontrivial if they are not induced by global isometries. Our choice provides a proper Lie-group action, which we extensively use throughout our theoretical results.

In the Euclidean setting, the infinitesimal theory is developed systematically by Asimow and Roth \cite{AsimowRoth1978,AsimowRoth1979}. The authors define the rigidity matrix as the differential of the map of squared edge lengths. The rigidity operator given by the differential of the squared geodesic edge lengths is therefore the natural analogue to the rigidity matrix on Riemannian manifolds. We can study its kernel to develop an infinitesimal rigidity theory for intrinsic frameworks on Riemannian manifolds that serves as a simple sufficient condition for the framework's rigidity. 

However, infinitesimal rigidity is not a necessary condition. For that reason, we develop a second-order theory (e.g.\ \cite{connelly1996}) for the rigidity of intrinsic frameworks. In total, we obtain the following implications for the rigidity of intrinsic bar-and-joint frameworks, which are completely analogous to the Euclidean setting.
\[
\boxed{
	\text{infinitesimally rigid}
	\Longrightarrow
	\text{prestress stable}
	\Longrightarrow
	\text{second-order rigid}
	\Longrightarrow
	\text{locally rigid}.
}
\]

\subsection{Contributions}
In this article, we develop a coordinate-free formulation of intrinsic bar-and-joint rigidity on smooth Riemannian manifolds. Unlike frameworks constrained to embedded surfaces with ambient Euclidean edge lengths \cite{NixonOwenPower2012} or special surfaces with induced geodesic distances that require an individual treatment \cite{NixonPower2020,SaliolaWhiteley2007}, our framework constraints are given by intrinsic Riemannian distances on arbitrary Riemannian manifolds. This allows us to treat manifolds with general Riemannian metrics without requiring an ambient Euclidean embedding or a global linear structure, as is available in the normed-space settings of \cite{dewarnormed,dewarnormed2}. Our main contribution is to establish an intrinsic differential-geometric notion of classical first- and second-order rigidity theory on Riemannian manifolds, with particular emphasis on the explicit dependence of the rigidity operator and second-order stress forms on the existence of geodesics.

First, we construct an intrinsic rigidity operator as the differential of the squared geodesic-distance constraint map (Definition \ref{def:rigidity-operator}). Using the first-variation formula, we obtain an explicit representation of its edge rows in terms of the Riemannian logarithm (Proposition \ref{prop:edge-row}). This provides a coordinate-independent analogue of the Euclidean rigidity matrix and permits us to characterize infinitesimal rigidity modulo the infinitesimal action of the isometry group (Definition \ref{def:inf_rigidity}). We derive a Maxwell-type necessary count (Corollary \ref{cor:maxwell}) and establish the relationship between infinitesimal and local rigidity: the two notions are equivalent at regular placements (Proposition \ref{prop:regular-equivalence}), while infinitesimal rigidity implies local rigidity without a regularity assumption (Theorem \ref{thm:inf-rigid}). We further investigate the relationship between flexibility and continuous motions, proving the existence of nontrivial smooth motions at regular placements (Theorem \ref{thm:existenceofcontmotion}) and nontrivial real-analytic motions on real-analytic Riemannian manifolds (Theorem \ref{thm:existencecontinuousmotionsrealanalytic}).

Second, we develop a covariant second-order rigidity theory that accounts for the variation of tangent spaces and the Riemannian metric. The second-order edge constraints are expressed through Riemannian Hessians and covariant accelerations (Proposition \ref{prop:curve-second-derivative}). In particular, we derive an intrinsic representation of the edge Hessians using Jacobi fields along the minimizing geodesics (Proposition \ref{prop:riemannian-edge-hessian}), making the dependence of the second-order theory on the Riemannian geometry explicit. We then introduce self-stresses and their associated quadratic forms (Definition \ref{def:self-stress}), prove that these forms descend to the space of infinitesimal flexes modulo trivial motions (Lemma \ref{lem:radical}), and obtain an exact characterization of second-order rigidity through self-stresses (Theorem 5.8 and Corollary 5.9). This yields intrinsic notions of prestress stability and second-order rigidity, together with the implication hierarchy from infinitesimal rigidity to local rigidity (Theorems \ref{thm:inf-prestress}--\ref{thm:so-local}). Finally, we recover the familiar Euclidean, spherical, and hyperbolic rigidity operators as special cases and illustrate the second-order theory through frameworks on manifolds equipped with both standard and nonstandard Riemannian metrics (Sections \ref{section:specialization-examples}).

Our approach builds on the differential approach from classical rigidity theory and its extensions to normed spaces \cite{dewarnormed,dewarnormed2}. Several of the theoretical implications established in this article, particularly the relationships between infinitesimal, second-order, and local rigidity, reflect general principles for smooth constraint maps invariant under proper Lie-group actions. The main novelty therefore lies not in these implications themselves, but in their intrinsic realization for geodesic-distance constraints on arbitrary Riemannian manifolds. More precisely, we identify the differential of the squared geodesic-distance map through the Riemannian logarithm and express its second-order derivatives in terms of the Riemannian Hessian and Jacobi fields along minimizing geodesics. These formulas make the dependence of rigidity and prestress stability on the underlying Riemannian geometry explicit, without requiring an ambient Euclidean embedding or a global linear structure. Furthermore, by defining trivial motions through the global isometry group, our formulation accounts for the distinction between local and global isometries, which becomes relevant on general Riemannian manifolds. The resulting framework thus provides a unified geometric setting in which classical rigidity criteria can be formulated intrinsically and applied to manifolds with general Riemannian metrics.

\section{Riemannian Manifolds}\label{section:rie_man}

In this section, we provide a readable overview of the concepts from Riemannian geometry that are relevant for this article, without giving too many details. For technical derivations and results, see the Appendix Section \ref{section:smooth-manifolds} on smooth manifolds and Section \ref{app:riemannian manifolds} on Riemannian manifolds, which are supposed to provide sufficient background knowledge for this paper. Alternatively, standard textbook references are \cite{boumal2020intromanifolds, JostRieGeo, kobayashinomizu, Lee2018, Lee-manifolds}.

 Given a smooth manifold $\M$, we can associate the \struc{tangent space} $T_p\M$ to each point $p\in \M$, which is given by equivalence classes of smooth curves on $\M$ through $p$ (Definition \ref{def:tangent-space-and-local-frame}). We can select smooth tangent vector fields along $\M$ that define a smoothly varying basis for each tangent space $T_p\M$, which we call a \struc{local frame} (Definition \ref{def:local_frame}). Given a smooth function $f:\M\rightarrow \mathbb{R}$, the \struc{differential} $df_p:T_p\M\rightarrow \mathbb{R}$ of $f$ at $p$ provides a natural differentiable structure on smooth manifolds (Definition \ref{def:differential}).

Riemannian geometry provides the framework for studying intrinsic distances, angles and areas on smooth manifolds. A pair $(\M,g)$ consisting of a smooth manifold $\M$ and a smoothly-varying bilinear form $g$ is called a \struc{Riemannian manifold} (Definition \ref{def:riemannian-manifold}). The \struc{Riemannian metric} $g$ is defined at every point $p\in\M$ as an inner product $g_p:T_p\M\times T_p\M\rightarrow \mathbb{R}$ on the tangent space $T_p\M$ of $\M$. We can associate a \struc{Riemannian norm} $\|x\|_p=\sqrt{g_p(x,x)}$ with it. Every smooth manifold admits a Riemannian metric \cite[Prop.\ 2.4]{Lee2018}, and embedded submanifolds inherit a Riemannian metric through the pullback of the ambient metric (Definiton \ref{def:riemannian-submanifold}).

We consider smooth curves $\gamma:[a,b]\rightarrow \M$ on Riemannian manifolds, which come with the canonical \struc{velocity vector field} $\gamma'(t)\in T_{\gamma(t)}\M$. This construction provides an alternative characterization of the tangent space at $p$ in terms of tangents of smooth curves through $p$ (Proposition \ref{prop:tangent_space}). The metric uniquely determines the \struc{Levi-Civita connection} $\nabla$, which allows for covariant differentiation of vector fields. The covariant derivative also lets us define \struc{geodesics} as curves $\gamma$ which have zero instrinsic acceleration: $\nabla_{\gamma'}\gamma'=0$ (Definition \ref{def:geodesics}). The \struc{Riemannian distance} (Definition \ref{def:riemannian-distance}) between two points is given by the infimum of the lengths of admissible curves $\gamma$:
\[
\ell_g(\gamma(a),\,\gamma(b))=\inf_{\gamma}\int_a^b\|\gamma'(t)\|_{\gamma(t)} \,dt.
\]
Length-minimizing curves are geodesics, although geodesics need not be globally minimizing \cite[Thm.\ 6.4 and 6.15]{Lee2018}. 

The exponential map $\mathrm{Exp}_p:T_p\M\rightarrow \M$ associates tangent vectors in $T_p\M$ with endpoints of geodesics starting at the point $p$ (Definition \ref{def:exponential-map}). For tangent vectors within the \struc{injectivity radius} $\mathrm{inj}(p)$ (Definition \ref{def:injectivity-radius}), this map is invertible, giving rise to the \struc{Riemannian logarithm} $\mathrm{Log}_p=\left(\Exp_p\right)^{-1}:\mathcal{U}\rightarrow T_p\M$ for a sufficiently small neighborhood $\mathcal{U}\subset \M$ of $p$ (Definition \ref{def:riemannian-logarithm}) and the relation $\ell_g(p,q)=\|\mathrm{Log}_p(q)\|_p$ \cite[Prop.\ 10.22]{boumal2020intromanifolds}. \struc{Parallel transport} provides metric-preserving identifications of tangent spaces along curves (Definition \ref{def:parallel-transport}), while \struc{Jacobi fields} describe geodesic variations and their dependence on curvature (Definition \ref{def:jacobi-field}). These objects will prove useful when also considering second-order information on Riemannian manifolds, such as the \struc{Riemannian Hessian} (Definition \ref{def:hessian}).

Since they define the trivial modes of geometric objects on Riemannian manifolds, \struc{isometries} are particularly important for the rigidity of instrinsic frameworks on Riemannian manifolds. Namely, isometries are given by diffeomorphisms $\varphi:(\mathcal{M},g)\to(\mathcal{M},g)$ satisfying $\varphi^*g=g$ (Definition \ref{def:isometries}). Such maps preserve intrinsic distances and therefore provide a natural notion of congruence on Riemannian manifolds. The \struc{isometry group} $\mathrm{Isom}(\mathcal{M},g)$ is a Lie group \cite{MyersSteenrod1939}, whose Lie algebra $\mathfrak{isom}(\mathcal{M},g)$ describes infinitesimal isometries and has dimension at most $n(n+1)/2$ for $n=\dim\mathcal{M}$ \cite[Thm.\ 3.3]{kobayashinomizu}. Its elements $A\in\mathfrak{isom}(\M,g)$ generate smooth one-parameter families of isometries $\Phi_t^A$, providing the geometric foundation for distinguishing genuine deformations from motions induced by the ambient isometry group. For a connected Riemannian manifold, the group action of the Lie group of isometries on $\M$ is proper, which is a property that we rely heavily on throughout this article. For this reason, we will from now on assume that any Riemannian manifold that we encounter is connected (or that the vertices of a framework exclusively lie on a connected component of $\M$).

\section{Rigidity of intrinsic frameworks on Riemannian manifolds}\label{section:inf_theory}

Let $G=(V,E)$ be a finite simple graph with vertices $V=\{1,\dots,n\}$. We embed this graph on a Riemannian manifold $(\M,g)$ so that the vertices are constrained to $\M$ and the edges form minimizing geodesic arcs subject to the Riemannian metric $g$.

\begin{definition}[Intrinsic Riemannian framework]
An \struc{intrinsic Riemannian framework} is a pair $(G,p)$, where
\[
 p=(p_1,\dots,p_n)\in \M^n
\]
is a placement of the vertices, which we call a \struc{realization}.  We call the graph's edges the \struc{bars} of the framework. For each edge $ij\in E$, the bar length is given by the intrinsic Riemannian distance $\ell_g(p_i,p_j)$. 
\end{definition}

Thus, the bars constrain intrinsic distance, rather than the distance of an ambient embedding, as is the case in \cite{NixonOwenPower2012}.  An edge need not be regarded as an additional geometric object.  When a unique minimizing geodesic joins $p_i$ to $p_j$, that geodesic provides the intrinsic representative of the bar and is the object that appears in the first-variation formula.

Recall the set $\mathcal{U}_\mathrm{Log}$ consisting of points on $\M\times \M$ whose Riemannian distance is bounded above by the injectivity radius (see Equation \eqref{eq:U_log}).
On $\mathcal U_{ \mathrm{Log}}$, the Riemannian logarithm $\Log_x(y)$ is the initial velocity at $x$ of the unique minimizing geodesic from $x$ to $y$, parametrized on $[0,1]$. Hence,
\[
 \|\Log_x(y)\|_x=\ell_g(x,y) \quad \text{ and }\quad
 \Exp_x(\Log_x(y))=y
\]
(see \cite[Prop.\ 10.22]{boumal2020intromanifolds}). For that reason, the squared distance function $\ell_g(x,y)^2$ is smooth, since it is always attained by the unique minimizing geodesic between $x$ and $y$, which varies smoothly according to the exponential map.

In order to ensure that there exists a unique minimizing geodesic arc between any two vertices of the framework, we use this observation to introduce an additional assumption on the realization.

\begin{definition}[Admissible placement]\label{def:admissibility}
A placement $p\in \M^n$ is \struc{admissible for $G$} if
\[
 (p_i,p_j)\in\mathcal{U}_{\mathrm{Log}}\quad\text{ and }\quad (p_j,p_i)\in\mathcal{U}_{\mathrm{Log}}
 \qquad\text{for every }ij\in E.
\]
\end{definition}
Note that we require that both edge directions lie in the domain $\mathcal{U}_{\mathrm{Log}}$ where the Riemannian logarithm is defined, since we intend to compute both, $\Log_{p_i}(p_j)$ and $\Log_{p_j}(p_i)$. Alternatively, we could require the end points of edges in admissible placements to lie in the complement of the \struc{cut locus} of each other \cite[p.\ 310f.]{Lee2018}. The preimage of this set under the exponential map exactly captures the maximum region in the tangent bundle, where the exponential map is a diffeomorphism. However, its definition is quite technical, so we use the more natural injectivity radius instead. 

For an admissible placement define the \struc{constraint map}
\begin{equation*}
 F_G:\mathcal C_G\longrightarrow \R^{E},
 \qquad
 (F_G(q))_{ij}=\frac12 \ell_g(q_i,q_j)^2,
\end{equation*}
where $\mathcal C_G\subset \M^n$ is the open set of admissible placements $p$ for $G$. The \struc{realization space} of a framework $(G,p)$ then is the collection of all admissible placements that have the same edge lengths, i.e.\ ${F_G}^{-1}({F_G}(p))$.

\begin{remark}
One may instead use the unsquared edge-distance map $q\mapsto \ell_g(q_i,q_j)$ whenever adjacent vertices are distinct.  At such a placement, the two differentials differ edgewise by multiplication by the nonzero factor $\ell_g(p_i,p_j)$.  Therefore the two choices have the same kernel and the same rank.  The squared-distance normalization is more convenient because its gradient is the logarithm vector itself.
\end{remark}

\subsection{The Riemannian rigidity operator}

The tangent space of the configuration manifold is
\[
 T_p\M^n=\bigoplus_{i=1}^n T_{p_i}\M.
\]
In other words, any tangent vector of $\M^n$ decomposes as follows:
\[
 u=(u_1,\dots,u_n),
 \qquad u_i\in T_{p_i}\M.
\]
To investigate the rigidity of an intrinsic bar-and-joint framework, we take a look at the differential of the constraint map. It gives rise to an analogue of the rigidity matrix.

\begin{definition}[Intrinsic Riemannian rigidity operator]\label{def:rigidity-operator}
For an admissible framework $(G,p)$, the \struc{intrinsic Riemannian rigidity operator} is
\begin{equation*}\label{eq:rigidity-operator}
 R_\M(G,p):=d(F_G)_p:
 \bigoplus_{i=1}^n T_{p_i}\M\longrightarrow \R^E.
\end{equation*}
\end{definition}

The following formula provides a linear and edge-wise representation of the rigidity operator. 

\begin{proposition}[Edge row formula]\label{prop:edge-row}
Let $(G,p)$ be admissible.  For $u=(u_1,\dots,u_n)\in T_p\M^n$ and $ij\in E$,
\begin{eqnarray}\label{eq:row-formula}
 \bigl(R_\M(G,p)u\bigr)_{ij}
 &=&-\bigl\langle \Log_{p_i}(p_j),u_i\bigr\rangle_{p_i}
  -\bigl\langle \Log_{p_j}(p_i),u_j\bigr\rangle_{p_j}.
\end{eqnarray}
\end{proposition}

\begin{proof}
For fixed $y\in\M$ with $(x,y)\in \mathcal{U}_\mathrm{Log}$, the corresponding Riemannian logarithm is well-defined, smooth and the first variation of squared distance gives
\[
 \mathrm{grad}_x\frac12 \ell_g(x,y)^2=-\Log_x(y)
\]
(see \cite[Thm.\ 1.2]{Karcher1977}). Applying this identity to both variables of $\frac12\ell_g(q_i,q_j)^2$ and evaluating on $(u_i,u_j)$ yields Equation \eqref{eq:row-formula}.
\end{proof}

\begin{remark}
A particularly convenient representation of the rigidity operator is in the form of a matrix. For that, let $d=\dim \M$ and choose an orthonormal basis $e_{i1},\dots,e_{id}$
of $T_{p_i}\M$ for each vertex $i$. We express
\[
 \Log_{p_i}(p_j)=\sum_{k=1}^d \lambda_{ij}^{(k)}e_{ik}
\]
in terms of the chosen orthonormal basis. In the product basis of $T_p\M^n$, the operator $R_\M(G,p)$ is represented by an $|E|\times dn$ matrix.  The row indexed by $ij$ has only two nonzero $d$-blocks:
\begin{eqnarray*}
 \begin{pmatrix}
  0&\cdots&-(\lambda_{ij})^{\mathsf T}&\cdots
  &-(\lambda_{ji})^{\mathsf T}&\cdots&0
 \end{pmatrix}.
\end{eqnarray*}
Changing the orthonormal frame at a vertex right-multiplies the corresponding block column by an orthogonal matrix.  Thus the kernel dimension and rank are independent of these choices.
\end{remark}

\subsection{Infinitesimal flexes and trivial motions}\label{section:inf_flexes}
In classical rigidity theory, investigating the kernel of the rigidity operator gives rise to a computationally-tractable sufficient condition for the rigidity of the corresponding bar-and-joint framework. A similar construction also works in the intrinsic Riemannian setting. 

\begin{definition}[Infinitesimal flex]
An \struc{infinitesimal flex} of $(G,p)$ is a vector
\[
 u\in\ker R_\M(G,p).
\]
Equivalently, an infinitesimal flex $u$ satisfies the following equation for all edges $ij\in E$:
\begin{equation}\label{eq:edge-equation}
	\bigl\langle \Log_{p_i}(p_j),u_i\bigr\rangle_{p_i}
	+\bigl\langle \Log_{p_j}(p_i),u_j\bigr\rangle_{p_j}=0.
\end{equation}
\end{definition}

There is a useful equivalent expression in terms of parallel transport.  Let $\mathcal{T}_{i\to j}:T_{p_i}\M\to T_{p_j}\M$ denote the parallel transport along the unique minimizing geodesic from $p_i$ to $p_j$ (see \cite[p.\ 105f.]{Lee2018}).  Reversing that geodesic gives
\begin{eqnarray*}
	\Log_{p_j}(p_i)&=&-\mathcal{T}_{i\to j}\Log_{p_i}(p_j).
\end{eqnarray*}
Hence, Equation \eqref{eq:edge-equation} is equivalent to
\begin{eqnarray*}
	\left\langle
	\Log_{p_i}(p_j),~
	u_i-\mathcal{T}_{j\to i}u_j
	\right\rangle_{p_i}&=&0.
\end{eqnarray*}
This equation is the intrinsic analogue of the Euclidean statement that the relative velocity of the endpoints of an edge is orthogonal to the bar's direction.

At every point of $\M^n$, there are natural infinitesimal flexes which are generated by ambient isometry of $\M$, which is why they are called ``trivial''. Let $\mathfrak{isom}(\M,g)$ denote the Lie algebra of $\Isom(\M,g)$. For any element $A\in \mathfrak{isom}(\M,g)$, there exists a unique one-parameter subgroup $\Phi^A_t$ of $\textsc{Isom}(\M,g)$ that is given by the exponential map. The evaluation at the framework vertices defines
\[
 \operatorname{ev}_p:\mathfrak{isom}(\M,g)
 \longrightarrow T_p\M^n,
 \qquad
 A\longmapsto \left(A(p_1),\,\dots,\,A(p_n)\right).
\]
We define the set of \struc{trivial infinitesimal flexes $\cK_p$} as $\operatorname{im}(\operatorname{ev}_p)$. This lets us formulate the following proposition.

\begin{proposition}\label{prop:killing-kernel}
For every admissible framework $(G,p)$,
\[
 \cK_p\subseteq \ker R_\M(G,p).
\]
\end{proposition}

\begin{proof}
Consider the unique one-parameter family of isometries $\Phi^A_t\in \textsc{Isom}(\M,g)$ induced by some $A\in \mathfrak{isom}(\M,g)$ (see \Cref{section:rie_man}).  Since $\Phi^A_t$ is an isometry for each $t$, it preserves Riemannian distance, so
\[
 \ell_g(\Phi^A_t(p_i),\,\Phi^A_t(p_j))=\ell_g(p_i,\,p_j)
\]
for every pair of edges $i,j$.  Differentiating at $t=0$ shows that the vector $\mathrm{ev}_p(A)$
lies in the kernel of the rigidity operator, since $\frac{\partial}{\partial t} \Phi^A_t\big\lvert_{t=0}=A$. In other words, the tangent vectors obtained by differentiating ambient isometries always lie in the kernel of the intrinsic rigidity matrix.
\end{proof}

\begin{definition}[Infinitesimal rigidity]\label{def:inf_rigidity}
An admissible framework $(G,p)$ is \struc{infinitesimally rigid} if
\[
 \ker R_\M(G,p)=\cK_p.
\]
It is \struc{infinitesimally flexible} otherwise.
\end{definition}

Consequently, if $t_p:=\dim\cK_p$, then infinitesimal rigidity is equivalent to
\begin{equation}\label{eq:rank-criterion}
 \rank R_\M(G,p)=dn-t_p.
\end{equation}
The quantity $t_p$ is the dimension of the isometry orbit through the placement. Denote by $\textsc{Isom}_{\mathrm{id}}(\M,g)\subset \textsc{Isom}(\M,g)$ the connected component of $\textsc{Isom}(\M,g)$ containing the identity element $\mathrm{id}\in \textsc{Isom}(\M,g)$ and call it the \struc{identity component}. If the identity component of $\Isom(\M,g)$ has dimension $k$ and the stabilizer of the entire placement is discrete, then $t_p=k$.

\begin{corollary}[Maxwell-type necessary count]\label{cor:maxwell}
If $(G,p)$ is infinitesimally rigid, then
\begin{eqnarray*}
 |E|\ge dn-t_p.
\end{eqnarray*}
In particular, if $\dim\Isom_{\mathrm{id}}(\M,g)=k$ and the placement has discrete stabilizer, then
\begin{eqnarray*}
 |E|\ge dn-k.
\end{eqnarray*}
\end{corollary}

\begin{proof}
The rank of an $|E|\times dn$ matrix is at most $|E|$.  Then, apply \eqref{eq:rank-criterion}.
\end{proof}

As the final part of this chapter, let us consider an example to illustrate how infinitesimal rigidity behaves on Riemannian manifolds. 

\begin{example}\label{ex:hyperbolictriang}
	Consider a 3-cycle graph in the upper half-plane model $\mathbb{H}\subset \mathbb{R}^2$ of the hyperbolic plane, so that $p_1$ is placed at $(0,1)$, $p_2$ is placed at $(1,1)$ and $p_3$ is placed at $(\frac{1}{2},\frac{1}{2})$. This triangle is depicted in \Cref{fig:triangandsaddle}(l.). Using the Riemannian metric tensor $ds^2=\frac{d{x_1}^2+d{x_2}^2}{{x_2}^2}$, $\mathbb{H}$ becomes a Riemannian manifold. 
	
	The geodesics are given by circular arcs and straight lines that are perpendicular to the line $\mathbb{R}\cdot (1,0)$ and the geodesic distance is characterized by the distance formula (cf.\ \cite[p.\ 81f.]{geometrie})
	\[\cosh d_\mathbb{H}(p,q)\,=\,1+\frac{\|p-q\|^2}{2\cdot p_2\cdot q_2}.\] 
	The hyperbolic triangle constructed above is therefore equilateral and has the edge length $\mathrm{arccosh}\left(\frac{3}{2}\right)$. Now, the group of isometries consists of horizontal translations, dilations and inversions. This group is isomorphic to $\mathrm{PGL}(2,\mathbb{R})$, which has dimension 3. Therefore, Corollary \ref{cor:maxwell} predicts that the necessary Maxwell count is exactly satisfied and the intrinsic bar-and-joint framework has the potential to be infinitesimally rigid.
	
	Using the tangential information (see Proposition \ref{prop:edge-row}) of the geodesic arcs that make up the edges of the intrinsic 3-cycle framework, i.e.\ segments of the three circles $\|p-(m_1,0)\|^2=\frac{5}{4}$ for $m_1\in \{-\frac{1}{2},\,\frac{1}{2},\,\frac{3}{2}\}$, we obtain the following matrix representation of the rigidity operator:
	\[\frac{\ell}{\sinh(\ell)}\cdot \begin{pmatrix}
		-1 & -\frac{1}{2} & 1 & -\frac{1}{2} & 0 & 0\\
		-1 & \frac{1}{2} & 0 & 0 & 1 & -2\\
		0 & 0 & 1 & \frac{1}{2} & -1 & -2
	\end{pmatrix}.\]
	for $\ell=\mathrm{acosh}(3/2)$. Since this matrix has rank 3, its kernel is 3-dimensional, which implies that all infinitesimal flexes corresponding to the 3-cycle are trivial. Therefore, the 3-cycle is infinitesimally rigid according to Definition \ref{def:inf_rigidity}.
\end{example}

\begin{figure}[h!]
	\centering
	\begin{minipage}[c]{0.275\linewidth}
		\includegraphics[width=\linewidth]{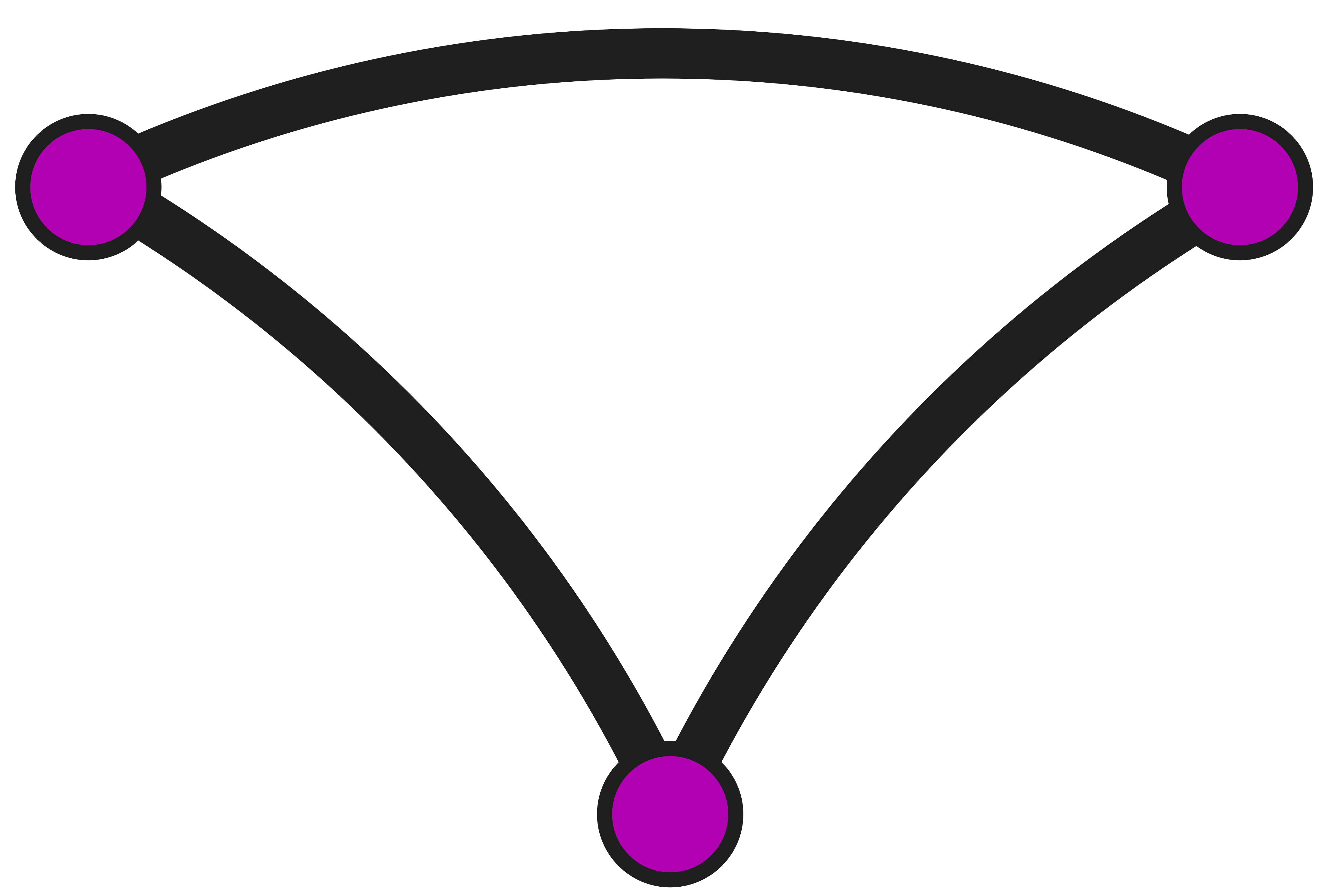}
	\end{minipage}\hspace*{26mm}
	\begin{minipage}[c]{0.375\linewidth}
		\includegraphics[width=\linewidth]{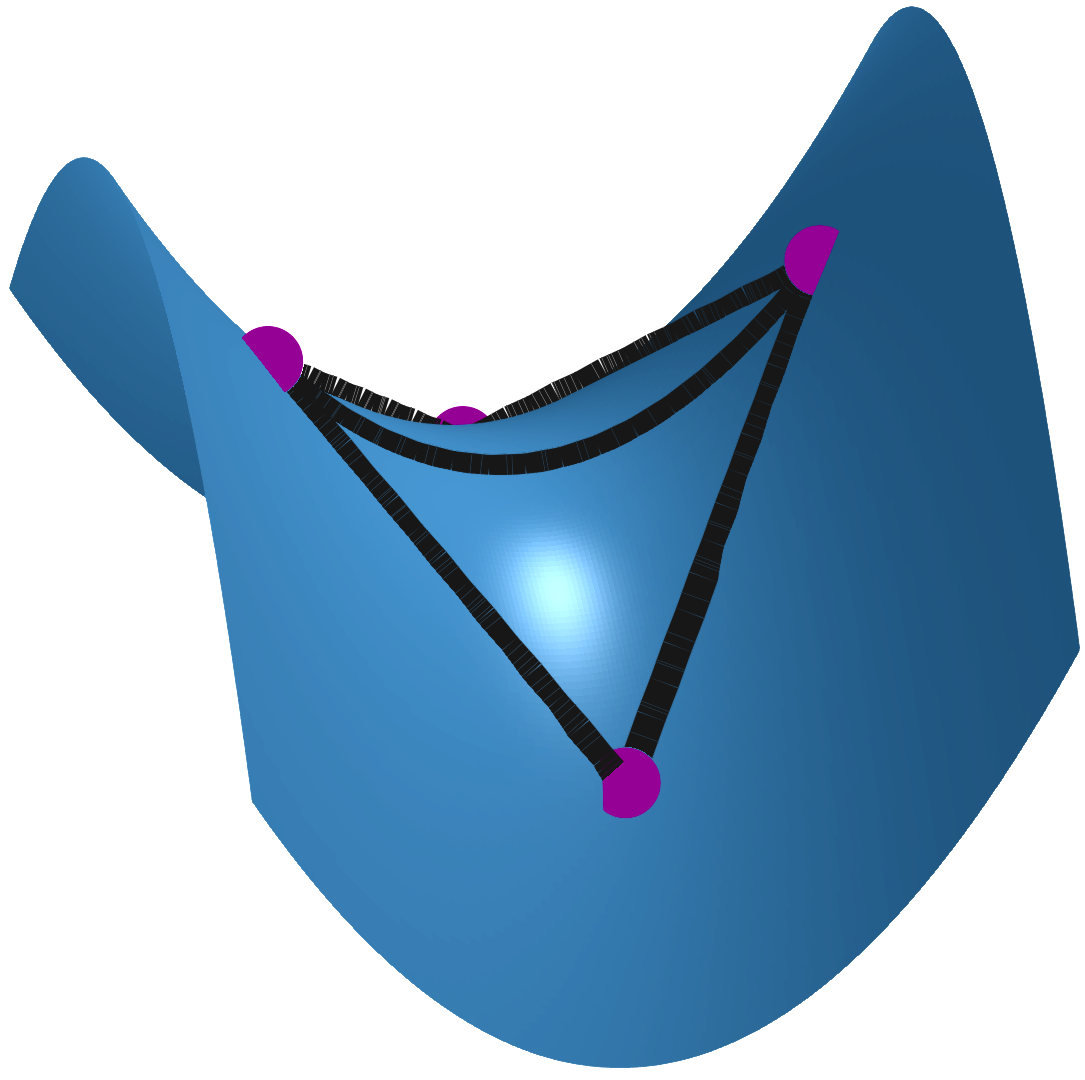}
	\end{minipage}
	\caption{\textbf{(l.)} An equilateral hyperbolic triangle in the upper half-plane model and \textbf{(r.)} a diamond graph on the saddle surface.}
	\label{fig:triangandsaddle}
\end{figure}

Next, we consider an example that is not infinitesimally rigid.

\begin{example}\label{ex:saddlediamond}
	Consider the saddle surface $\mathcal{S}$ cut out by the implicit equation $z=x^2-y^2$ together with the induced Euclidean metric. We take the diamond graph whose vertices are placed at $p_1=(1,0,1)$, $p_2=(0,1,-1)$, $p_3=(-1,0,1)$ and $p_4=(0,-1,-1)$. This intrinsic framework is depicted in Figure \ref{fig:triangandsaddle}(r.). Since $\mathcal{S}$ is a ruled surface, the geodesic arcs on the diamond graph's edges $(1,2)$, $(2,3)$, $(3,4)$ and $(1,4)$ are given by straight lines, while the geodesic arc on the edge $(1,3)$ is given by the intersection of $\mathcal{S}$ with a linear subspace. 
	
	The isometry group has dimension $0$; the surface only has discrete Euclidean symmetries such as reflections. However, the diamond graph only has 5 edges, which is smaller than $2\cdot 4=8$, implying that the necessary Maxwell condition from Corollary \ref{cor:maxwell} is not satisfied and the intrinsic framework is not infinitesimally rigid.
\end{example}

\section{Local rigidity and regular placements}\label{section:local_rigidity}

Recall that the identity component $\textsc{Isom}_{\mathrm{id}}$ is the connected component of $\textsc{Isom}(\M,g)$ that contains the identity $\mathrm{id}$. We define the orbit of the realization $p$ under the identity component as the orbit
\begin{eqnarray}\label{eqn:isometry-orbit}
 \cO_p=\Isom_{\mathrm{id}}(\M,g)\cdot p \subset \M^n.
\end{eqnarray}
We call $\mathcal{O}_p$ the \struc{trivial motions} of the framework. Since the group action preserves all pairwise Riemannian distances, it holds that
\[
 \cO_p\subseteq F_G^{-1}(F_G(p)).
\]

\begin{definition}[Local rigidity]\label{def:local_rigidity}
An admissible framework $(G,p)$ is \struc{(locally) rigid} if there exists a neighborhood $U$ of $p$ in $\M^n$ such that
\[
 F_G^{-1}(F_G(p))\cap U
 \subseteq \cO_p.
\]
It is called \struc{flexible} otherwise.
\end{definition}

This definition formally uses the ambient isometries as the trivial motions.  In a general metric-space formulation one may instead declare a motion trivial when it preserves all pairwise distances of the given finite configuration; these notions need not coincide on every Riemannian manifold.  See \cite{NixonPower2020} for the corresponding distinction in general metric contexts.

\begin{remark}\label{rem:local_isometry_group}
	 In Definition \ref{def:local_rigidity}, trivial motions are induced by the identity component of the global isometry group. This differs from local isometries, which are metric-preserving diffeomorphisms between open subsets of $\M$ and need not extend globally.
	
	For example, consider the interval $\M=(0,1)$ with the Euclidean metric. Translating all vertices of a framework by a sufficiently small amount preserves every intrinsic edge length. However, since the isometry group $\operatorname{Isom}_{\mathrm{id}}(M,g)$ is trivial, these motions are nontrivial according to Definition \ref{def:local_rigidity}. This phenomenon can occur even for complete graphs at regular placements.
	
	One could alternatively define trivial motions using local isometries \cite[p.\ 13]{Lee2018}. However, these generally form a pseudogroup rather than a global Lie-group action. Since we extensively use this property throughout this article, we therefore retain global isometries, whose proper action permits the application of the Slice Theorem in Sections \ref{section:local_rigidity} and \ref{section:second-order-theory}. Consequently, all rigidity notions in this article are to be understood modulo global ambient isometries of $\M$.
\end{remark}

Let us now consider an additional assumption on the placement $p$ that helps us to prove the coming results.

\begin{definition}[Regular placement]
An admissible placement $p$ is \struc{regular for $G$} if $\rank d(F_G)_q$ is constant for $q$ in some neighborhood of $p$.
\end{definition}

\begin{proposition}[Local equivalence]\label{prop:regular-equivalence}
Let $(G,p)$ be admissible and regular. Then $(G,p)$ is rigid if and only if it is infinitesimally rigid.
\end{proposition}

\begin{proof}
Set $\mathcal{L}:={F_G}^{-1}(F_G(p))$. The regularity of the realization $p$ implies that $\mathcal{L}$ is a smooth manifold near $p$ with tangent space $T_p\mathcal{L}=\ker R_\M(G,p)$ according to the Constant Rank Theorem. The trivial motions form a smooth immersed submanifold of this level set, with tangent space $T_p\cO_p=\cK_p$. In fact, $\mathcal{O}_p$ is an embedded submanifold of $\M^n$ and hence, $\mathcal{O}_p$ is locally embedded in $\mathcal{L}$.
Assume now that $(G,p)$ is infinitesimally rigid, implying that 
\[T_p\mathcal{O}_p=\mathcal{K}_p=\ker R_\M(G,p)=T_p\mathcal{L}.\] 
Since $\mathcal{O}_p\subseteq \mathcal{L}$, the inclusion $\iota:\mathcal{O}_p\hookrightarrow\mathcal{L}$ is smooth and the differential $d\iota_p:T_p\mathcal{O}_p\rightarrow T_p\mathcal{L}$ is an isomorphism. The Inverse Function Theorem therefore says that $\iota$ maps a neighborhood of $p$ in $\mathcal{O}_p$ diffeomorphically into an open neighborhood of $p$ in $\mathcal{L}$. Thus, after sufficiently shrinking the ambient neighborhood $U\subset \mathcal{L}$ it holds that $\mathcal{L}\cap U = \mathcal{O}_p\cap U$,
so $(G,p)$ is rigid.

Conversely, suppose that $(G,p)$ is rigid, so that \[{F_G}^{-1}(F_G(p))\cap U = \mathcal{O}_p\cap U\] for a sufficiently small open neighborhood $U$ of $p$. The manifolds $\mathcal{O}_p$ and $\mathcal{L}={F_G}^{-1}(F_G(p))$ therefore have the same tangent space at $p$ and hence
\[\ker R_\M (G,p)=T_p\mathcal{L}=T_p\mathcal{O}_p=\mathcal{K}_p.\]
Thus, $(G,p)$ is infinitesimally rigid.
\end{proof}

The regularity hypothesis is essential for the converse between finite and infinitesimal rigidity.  At singular placements a framework may be locally rigid while admitting nontrivial infinitesimal flexes, as in the classical Euclidean theory.

\begin{theorem}\label{thm:inf-rigid}
	The infinitesimal rigidity of a bar-and-joint framework $(G,p)$ in the intrinsic metric of a Riemannian manifold $(\mathcal{M},g)$ implies its rigidity.
	\begin{proof}
		Assume for the contraposition that $(G,p)$ is flexible. The identity component of the isometry group $\textsc{Isom}_{\mathrm{id}}(\M,g)$ acts properly and diagonally on $\M^n$, since $\M$ is connected by assumption. Therefore, by the Slice Theorem \cite[Prop.\ 2.3.1]{palaisslice} there exists a local slice $\mathcal{S}\subset\M^n$ through $p$ such that 
		\[T_p\M^n\,=\,T_p\mathcal{O}_p\oplus T_p\mathcal{S}\,=\,\mathcal{K}_p\oplus T_p\mathcal{S}.\]
		After shrinking $\mathcal{S}$ if necessary, every realization $q\in\mathcal{M}^n$ in a sufficiently small neighborhood $U\subset \mathcal{M}^n$ of $p$ can be written as $q=h\cdot s$ for some $h\in \textsc{Isom}_{\mathrm{id}}(\M,g)$ and $s\in \mathcal{S}$.

		Assume now that $(G,p)$ was infinitesimally rigid. In other words, $\ker R_\mathcal{M}(G,p)=\mathcal{K}_p$. Since $R_\M(G,p)=d(F_G)_p$, the restricted differential $d(F_G\big\lvert_\mathcal{S})$ satisfies
		\[\ker d(F_G\big\lvert_\mathcal{S})_p=\ker d(F_G)_p\cap T_p\mathcal{S}=\mathcal{K}_p\cap T_p\mathcal{S}=\{0\},\]
		which implies that $d(F_G\big\lvert_\mathcal{S}):T_p\mathcal{S}\rightarrow \mathbb{R}^E$ is injective. Thus, there exists a linear $\pi:\mathbb{R}^E\rightarrow \mathbb{R}^{\dim \mathcal{S}}$ such that 
		\[d(\pi\circ F_G\big\lvert_\mathcal{S})_p:T_p\mathcal{S}\rightarrow \mathbb{R}^{\dim \mathcal{S}}\]
		is an isomorphism. 
		Hence, by the inverse function theorem $\pi\circ F_G\big\lvert_S$ is injective after shrinking $\mathcal{S}$ around $p$. In particular, the equation $F_G(s)=F_G(p)$ implies $s=p$ for any $s\in \mathcal{S}$. 
		
		Now, let $q\in U$ satisfy $F_G(q)=F_G(p)$. Since $F_G$ is invariant under isometries, we obtain
		\[F_G(s)=F_G(h\cdot s)=F_G(q)=F_G(p)\]
		for some $s\in \mathcal{S}$. By the local injectivity established above, it follows that $s=p$. Therefore, $q=h\cdot s\in \mathcal{O}_p$ for all $q\in U$ and consequently, ${F_G}^{-1}(F_G(p))\cap U \subseteq \mathcal{O}_p$, so $(G,p)$ is locally rigid. This contradicts the assumption that the framework is flexible. Therefore, $\ker R_\mathcal{M}(G,p)\neq \mathcal{K}_p$ and the framework is infinitesimally flexible.
	\end{proof}
\end{theorem}

The converse of this implication is generally not true. For that, let us consider the following example.

\begin{example}\label{ex:hyperboloid}
	Consider a 3-cycle graph realized in a special position on the one-sheeted hyperboloid $\mathcal{H}$ in $\mathbb{R}^3$ cut out by the equation $x^2+y^2-z^2=1$ and equipped with the induced Riemannian metric $g$ from the ambient Euclidean space. The vertices of the corresponding framework are placed in the points $p_1=(1,0,0)$, $p_2=(\cos(\frac{2\pi}{3}), \,\sin(\frac{2\pi}{3}),\,0)$ and $p_3=(\cos(\frac{4\pi}{3}), \,\sin(\frac{4\pi}{3}),\,0)$ so that all vertices are adjacent. This framework is depicted in Figure \ref{fig:hyperboloid_square}
	
	\begin{figure}[h!]
		\centering
		\includegraphics[width=0.35\linewidth]{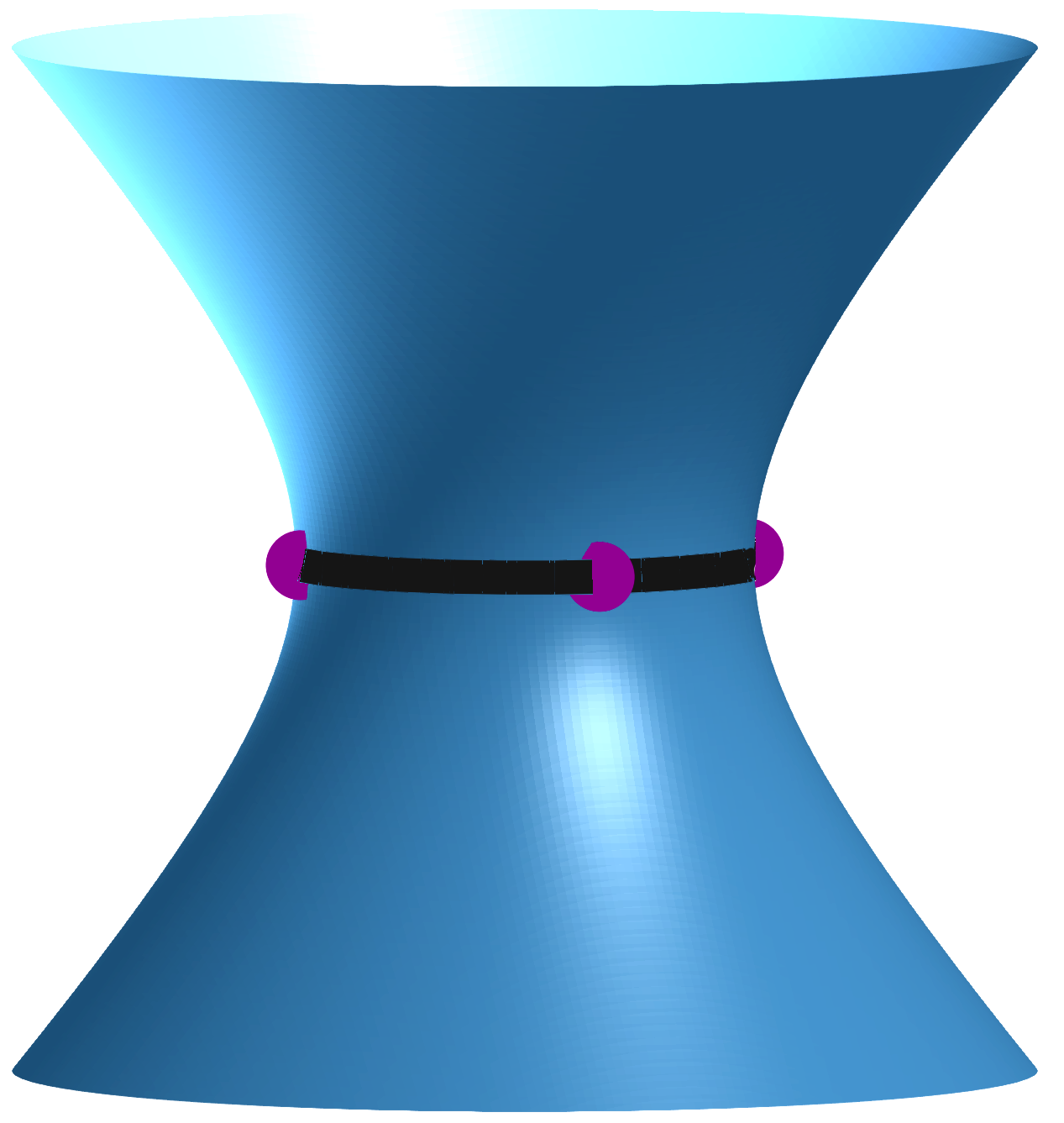}
		\caption{A 3-cycle graph on the one-sheeted hyperboloid. }
		\label{fig:hyperboloid_square}
	\end{figure}
	
	As a surface of revolution, the dimension of the hyperboloid's isometry group is 1 \cite{NixonOwenPower2012}. Therefore, the Maxwell count (cf.\ Corollary \ref{cor:maxwell}) is deficient and the framework is infinitesimally flexible. In fact, the kernel of the rigidity operator is generated by 4 vectors, leaving 3 nontrivial infinitesimal flexes. For instance, one such flex is given by applying the infinitesimal translation vector $(0,0,1)$ to each vertex. 
	
	Still, the framework is rigid. To demonstrate the claim, let us choose angular coordinates: $r=\sqrt{1+z^2}$ and $\theta\in [0,2\pi)$ so that the one-sheeted hyperboloid is given by 
	\[\mathcal{H}=\left\{\left(\sqrt{1+z^2}\cdot \cos(\theta),\, \sqrt{1+z^2}\cdot \sin(\theta),\,z\right)~ \bigg\lvert ~z\in\mathbb{R}\text{ and } \theta\in [0,2\pi)\right\}.\]
	Let us now consider any continuously differentiable curve \[\gamma:[0,1]\rightarrow \mathcal{H},\quad\gamma(t)=\left(\sqrt{1+z(t)^2}\cdot\cos(\theta(t)),\, \sqrt{1+z(t)^2}\cdot\sin(\theta(t)),\,z(t)\right)\] 
	Its Riemannian curve length satisfies
	\begin{eqnarray*}
	L_g(\gamma)&=&\int_{0}^{1} \sqrt{\frac{(z(t)\cdot z'(t))^2}{1+z(t)^2}+(1+z(t)^2)\cdot \theta'(t)^2+z'(t)^2}~dt \\
	&\geq & \int_{0}^{1} \sqrt{1+z(t)^2}\cdot \lvert \theta'(t)\lvert \,dt~\geq ~\int_{0}^{1} \lvert \theta'(t)\lvert \,dt
	\end{eqnarray*}
	In other words, the Riemannian distance between two points $p,q$ is bounded below by the absolute angular difference between $p$ and $q$. Returning to this chain of inequalities, we realize that for the length of geodesic arcs, equality to the absolute angular difference can only occur when $z'(t)=z(t)=0$. This is only the case when the curve lies on the hyperboloid's equator. 
	
	Now, the three intrinsic edges of the 3-cycle each have arc length $\frac{2\pi}{3}$. Suppose now that $q_1,q_2,q_3$ is a sufficiently small perturbation that preserves the edge lengths, i.e.\ $\ell_g(q_1,q_2)=\ell_g(q_2,q_3)=\ell_g(q_1,q_3)=\frac{2\pi}{3}$, and let 
	\[\theta_{i,j}=\arccos\left(\frac{\big\langle (q_{i,1},q_{i,2}),~(q_{j,1},q_{j,2})\big\rangle}{\sqrt{q_{i,1}^2+q_{i,2}^2}\cdot \sqrt{q_{j,1}^2+q_{j,2}^2}}\right)\] 
	be the angular separation between the vertices $q_i$ and $q_j$. By the admissibility of the realization $p$, it satisfies $\theta_{i,j}<\pi$ for all sufficiently close realizations $q$. By a previous argument, it holds that $\theta_{i,j}\leq \ell_g(q_i,q_j)$ for each $1\leq i < j \leq 3$. Since the cyclic order of the 3-cycle is unchanged by sufficiently small perturbations, it further holds that $\theta_{1,2}+\theta_{2,3}+\theta_{1,3}=2\pi$. Together, these observations imply that $\theta_{1,2}=\theta_{2,3}=\theta_{1,3}=\frac{2\pi}{3}$.
	
	However, this forces $q_1$, $q_2$ and $q_3$ to lie on the equator. Since the edge lengths are preserved by this perturbation, this implies that the deformed framework can only be obtained by a rotation of the original framework, which is an isometry of $\mathcal{H}$. Therefore, the 3-cycle is rigid, but not infinitesimally rigid, on $\mathcal{H}$.
\end{example}

\subsection{Flexibility and continuous motions} 
Dual to a frameworks rigidity is its flexibility: determining either of these concepts automatically gives a negative answer to the other. In this article, a framework's flexibility is defined in the usual sense by using accumulation points. With the Curve Selection Lemma (cf.\ \cite{curveselection}), the classical literature is able to infer the existence of continuous curves that do not come from isometry orbits (e.g.\ \cite{AsimowRoth1978}). However, this correspondence is generally not true in our setting, which we outline in this section.

First, we define the concept of a continuous motion, which are exactly continuous curves in the framework's realization space.

\begin{definition}
	A \struc{continuous/smooth/analytic motion} of $(G,p)$ is a continuous/smooth/analytic curve $\gamma:[0,1]\rightarrow {F_G}^{-1}(F_G(p))$ so that $\gamma(0)=p$. It is called \struc{trivial} if $\gamma(t)\in \mathcal{O}_p$ for all $t\in[0,1]$ and \struc{nontrivial} otherwise. 
\end{definition}

\begin{theorem}[Existence of motions in regular points]
	\label{thm:existenceofcontmotion}
	Let $(G,p)$ be admissible and regular. If $(G,p)$ is flexible, then there exists a smooth nontrivial motion in ${F_G}^{-1}(F_G(p))$ through $p$.
\end{theorem}
\begin{proof}
	By the Constant Rank Theorem, $\mathcal{L}:={F_G}^{-1}(F_G(p))$ is a smooth submanifold near $p$ and $T_p\mathcal{L}=\ker R_\M(G,p)$. According to Proposition \ref{prop:regular-equivalence}, the flexibility of $(G,p)$ implies that the framework is infinitesimally flexible. Therefore, we can choose $u\in \ker R_\M (G,p)\setminus \mathcal{K}_p = T_p\mathcal{L}\setminus \mathcal{K}_p$. By the definition of the tangent space, there exists an $\varepsilon>0$ and a smooth curve $\gamma:(-\varepsilon, \varepsilon)\rightarrow \mathcal{L}$ so that $\gamma(0)=p$ and $\gamma'(0)=u$. Since $\gamma(t)\in \mathcal{L}$, it further holds that $F_G(\gamma(t))=F_G(p)$ for all $t\in (-\varepsilon, \varepsilon)$, so that $\gamma$, after reparametrization, is a smooth motion of $(G,p)$. It is nontrivial, because $u\notin \mathcal{K}_p=T_p\mathcal{O}_p$, so $\gamma((-\varepsilon,\varepsilon))\not\subset \mathcal{O}_p$.
\end{proof}

Unlike in the classical algebraic setting, we cannot simply use the Curve Selection Lemma \cite{curveselection} to show that any flexible framework admits an analytic motion. The property that a flexible framework $(G,p)$ has accumulation points arbitrarily close to $p$ does not necessarily imply that we can find a curve through these accumulation points.

Consider for example the smooth curve
\[f:\mathbb{R}\rightarrow \mathbb{R},~~~~\begin{cases}
	e^{-\frac{1}{x^2}}\cdot \sin(\pi/x)& x\neq 0\\
	0& x=0
\end{cases}.\]
It holds that $f^{-1}(f(0))=f^{-1}(0)=\{0\}\cup\{\frac{1}{n}\,\lvert\, n\in \mathbb{Z}\setminus\{0\}\}$. Therefore, $0$ is not an isolated point of $f^{-1}(0)$. Nevertheless, there exists no nonconstant continuous curve $\gamma:[0,1]\rightarrow f^{-1}(0)$ with $\gamma(0)=0$, since the image of a connected set under a continuous map must stay connected, whereas every connected subset of $f^{-1}(0)$ consists of a single point. 

This example illustrates, why the assumption that $(G,p)$ is regular is essential in Theorem \ref{thm:existenceofcontmotion}. Nevertheless, by assuming that the Riemannian manifold $(\M,g)$ is analytic, we can use a more general version of the Curve Selection Lemma to obtain a similar result.

\begin{theorem}[Existence of motions on analytic manifolds]\label{thm:existencecontinuousmotionsrealanalytic}
	Assume that $(\M,g)$ is a real-analytic Riemannian manifold.
		If $(G,p)$ is flexible, then there exists a
	 real-analytic nontrivial motion through $p$.
	\end{theorem}
	
	\begin{proof}
		Since $p$ is admissible, every pair of vertices incident to an edge lies in $\mathcal{U}_\mathrm{Log}$.
		For a real-analytic Riemannian metric, the exponential map is real
		analytic, and therefore the logarithm map is locally real analytic
		wherever it is defined. While the Riemannian distance $\ell_g$ generally fails to be differentiable at $x=y$, the squared Riemannian distance ${\ell_g}^2$, and thus $F_G$, is real analytic
		in a neighborhood $U$ of $p$.
		Set
		\[
		\mathcal C=F_G^{-1}(F_G(p))\cap U.
		\]
		Then $\mathcal C$ is a real-analytic set. The isometry group acts diagonally and properly, as well as real-analytically on $\M^n$ (cf.\ \cite{baouendietal}), so the orbit $\mathcal O_p$ is
		locally an embedded real-analytic submanifold. In particular, $\mathcal{O}_p$ is locally defined by analytic coordinate equations. Its complement
		\[
		\mathcal{A}:=\mathcal C\setminus\mathcal O_p
		\]
		is therefore semi-analytic (cf.\ \cite{bierstone-millman}).
		
		Flexibility is characterized by the property that every neighborhood of $p$ contains a realization
		with the same edge lengths as $p$ which is not contained in $\mathcal O_p$ or, equivalently,
		$p\in\overline{\mathcal{A}}$, where $\overline{\mathcal{A}}$ denotes the closure of $\mathcal{A}$. We are now going to use the following technical lemma:
		\begin{lemma}[Real-analytic curve selection Lemma ({cf.\ \cite[Lem.\ 6.6]{realanalyticcurveselection}})]\label{lem:curve-selection-lemma}
			Let $\M$ be a real-analytic manifold and $X\subset\M$ a semi-analytic subset. If $p\in \overline{X\setminus\{p\}}$, then there exists $\varepsilon>0$ and a real-analytic curve $\gamma:[0,\varepsilon)\rightarrow \M$ such that $\gamma(0)=p$ and $\gamma((0,\varepsilon))\subset X\setminus\{p\}$. 
		\end{lemma}
		
		By Lemma \ref{lem:curve-selection-lemma}, there exist
		$\varepsilon>0$ and a real-analytic curve
		$\gamma:[0,\varepsilon)\to \M^n$
		such that
		\[
		\gamma(0)=p,
		\qquad
		\gamma(t)\in \mathcal{A}
		\quad\text{for }0<t<\varepsilon.
		\]
		Therefore,
		$
		F_G(\gamma(t))=F_G(p)
		$
		for all sufficiently small $t$ by construction, while
		$
		\gamma(t)\notin\mathcal O_p
		$ for $t>0$.
		In other words, $\gamma$ is a nontrivial real-analytic motion.
	\end{proof}

In fact, we can show that the assumption that the manifold is real-analytic is necessary for the existence of nontrivial motions.

\begin{proposition}
	There exist a connected, smooth Riemannian manifold $(\M,g)$ and an admissible
	intrinsic framework $(G,p)$ such that $(G,p)$ is locally flexible modulo the
	identity component of $\Isom(\M,g)$, but every continuous motion of $(G,p)$ is
	trivial.
\end{proposition}

\begin{proof}
	Let $\M=\mathbb R\times S^1$ be a cylinder with coordinates $(r,\theta)$ and angle $\theta\in[0,2\pi)$. Define the smooth function
	\begin{equation*}\label{eq:h}
		h(r)=\begin{cases}
			e^{-1/r^2}\sin(\pi/r)^2,&r>0,\\
			0,&r=0,\\
			e^{-1/r^2},&r<0.
		\end{cases}
	\end{equation*}
	Its zero set is the totally disconnected set $
	Z=h^{-1}(0)=\{0\}\cup\{1/n:n\in\mathbb N\}.$
	For $\varepsilon>0$, we define the Riemannian metric
	\begin{equation*}\label{eq:metric}
		g_\varepsilon=d r^2+(1+\varepsilon h(r))^2\,d\theta^2
	\end{equation*}
	on $\M$. Let $G$ be the 3-cycle with vertices $1,2,3$, and the realizations
	\begin{equation*}
	p_i=\left(0,~2\pi\cdot\frac{i-1}{3}\right)\qquad \text{and}\qquad
	q_i^{(n)}=\left(\frac{1}{n},~2\pi\cdot\frac{i-1}{3}\right).
	\end{equation*}
	For a piecewise $C^1$ curve $\gamma=(r,\theta)$, using any lift of its
	angular coordinate to $\mathbb R$, it holds that
	\begin{equation*}\label{eq:lengthbound}
		L_{g_\varepsilon}(\gamma)
		=\int\sqrt{\dot r^{\,2}+(1+\varepsilon h(r))^2\dot\theta^{\,2}}\,\dd t
		\geq\int|\dot\theta|\,\dd t.
	\end{equation*}
	If two points have the same height $r\in Z$ and angular separation less
	than $\pi$, the shorter horizontal circular arc attains this lower bound.
	It is the unique minimizing geodesic: equality forces $\dot r=0$ along
	any minimizing curve, and the shorter angular arc is unique.
	Consequently, all three edges of $p$ and of every $q^{(n)}$ have length
	$2\pi/3$, so
	\begin{equation}\label{eq:sequence}
		F_{G}(q^{(n)})=F_{G}(p),\qquad q^{(n)}\longrightarrow p.
	\end{equation}
	
	We first prove that there is no nontrivial continuous motion through $p$.
	Let \mbox{$q:[0,1]\to F_{G}^{-1}(F_{G}(p))$} be such a motion, where the
	realization space is restricted to admissible placements. Join the three
	successive vertices of $q(t)$ by their minimizing geodesic arcs, obtaining
	a closed curve $\Gamma_t$. Admissibility makes these arcs continuous in
	their endpoints. The winding number of $\Gamma_t$ around the cylinder is
	therefore constant and equals $1$, which is exactly its value at $t=0$.
	
	Each edge has length $2\pi/3$. A lift of $\Gamma_t$ to the universal cover
	thus satisfies
	\begin{equation}\label{eq:equality}
		2\pi=L_{g_\varepsilon}(\Gamma_t)
		\geq \int_{\Gamma_t}|d\theta|
		\geq \left|\int_{\Gamma_t}d\theta\right|=2\pi.
	\end{equation}
	This implies that $r$ is constant along every edge,
	hence all three vertices have one common height $r(t)$. It also implies
	$h(r(t))=0$, since the winding number of $\Gamma_t$ is exactly 1, so $r(t)\in Z$. Since $r(t)$ is continuous, $r(0)=0$, and since $Z$ is
	totally disconnected, it necessarily holds that $r(t)=0$ for all $t$.
	Moreover, equality in \eqref{eq:equality} forces all angular increments
	to have the same sign; as each edge has length $2\pi/3$ and the total
	increment is $2\pi$, each increment equals $2\pi/3$. Hence $q(t)$ differs
	from $p$ only by a common angular rotation. Such rotations are global
	isometries in the metric $g_\varepsilon$. Every continuous motion from $p$ is thus
	trivial.
	
	It follows that none of the distinct realizations $q^{(n)}$ belong to the
	$\Isom_{\mathrm{id}}(\M,g_\varepsilon)$-orbit of $p$:
	otherwise a one-parameter family of isometries $\Phi_t$ would yield a continuous
	motion from $p$ to $q^{(n)}$. Together with the convergence property from \eqref{eq:sequence}, this proves the
	local flexibility of $(G,p)$ in the sense of Definition \ref{def:local_rigidity}
	
	Finally, choose $\varepsilon>0$ sufficiently small to ensure admissibility
	in the sense of Definition \ref{def:admissibility}. Indeed, $h$ and its first two derivatives
	are bounded, so $g_\varepsilon$ converges uniformly in $C^2$ to the flat
	cylinder metric as $\varepsilon\to0$. On the universal cover, the
	exponential maps on tangent balls of any fixed radius $R<\pi$ then converge
	uniformly in $C^1$ to the induced Euclidean exponential maps. For sufficiently
	small $\varepsilon$ they remain injective on these balls, and distinct
	lifts cannot differ by a deck translation of angular size $2\pi$.
	Taking $2\pi/3<R$ gives an injectivity radius greater than $2\pi/3$
	at all the displayed vertices. The framework and nearby realizations
	are therefore admissible.
\end{proof}

\begin{remark}
	The metric $g_\varepsilon$ is smooth but not real analytic. Hence the
	example does not contradict the real-analytic curve-selection result from \Cref{thm:existencecontinuousmotionsrealanalytic}. It also shows that the regularity assumption in the
	smooth motion-existence theorem cannot simply be omitted: if $p$ were a
	regular placement, that theorem would produce a nontrivial smooth motion.
\end{remark}

\section{Intrinsic second-order rigidity theory}\label{section:second-order-theory}

Even in the Euclidean case, the infinitesimal theory does not completely characterize the rigidity of bar-and-joint frameworks.  In Example \ref{ex:hyperboloid}, we show that this assertions remains true in the Riemannian setting: a curve may preserve all edge lengths up to first order, but may fail to do so in higher orders.  This is the geometric mechanism behind a framework that is locally rigid, but infinitesimally flexible. Instead, second-order criteria (cf.\ \cite{connelly1996}) ask a broader question:

\begin{quote}
	Can an infinitesimal flex $u\in\ker R_\M(G,p)$ be consistently extended to a smooth curve, along which the Riemannian edge length is preserved up to second order?
\end{quote}

On a manifold, the answer has to be formulated covariantly. The second derivative of a curve in $\M^n$ is not an ordinary coordinate acceleration; it is its covariant acceleration with respect to the product Levi-Civita connection (see Definition \ref{def:geodesics}).

\subsection{The covariant second derivative of the constraint map}

For each edge $e\in E$, we denote the corresponding component of the constraint map by $f_e(p):=(F_G(p))_{e}$. We consider the Riemannian Hessian $\Hess f_e(p)$ of this scalar map on $\mathcal{M}^n$ and collect these edge Hessians into the symmetric bilinear map
\[
\mathcal H_p:T_p\mathcal{M}^n\times T_p\mathcal{M}^n\longrightarrow\mathbb R^E,
\qquad
\bigl(\mathcal H_p(u,v)\bigr)_e=\Hess f_e(p)[u,v].
\]
Using this notation, we can make the objects from the question above more precise:

\begin{proposition}[Second derivative along a curve]\label{prop:curve-second-derivative}
	Let $q(t)$ be a $C^2$ curve in $\mathcal{M}^n$ with
	\[
	q(0)=p,
	\qquad
	u:= q'(0),
	\qquad
	a:=\nabla^{\mathcal{M}^n}_{q'}q'\big|_{t=0}.
	\]
	Then,
	\begin{equation}\label{eq:second-derivative-F}
		\frac{\partial^2}{\partial t^2}F_G(q(t))\Big|_{t=0}
		=R_\M(G,p)a+\mathcal H_p(u,u).
	\end{equation}
\end{proposition}

\begin{proof}
	For a scalar function $f$ on a Riemannian manifold $\M$,
	\[
	\frac{\partial^2}{\partial t^2}f(q(t))\big\lvert_{t=0}
	=df_p\bigl(\nabla_{\dot q}\dot q\bigr)+\Hess f(p)[ q',q'].
	\]
	Then, we apply this identity componentwise to $F_G$ to obtain Equation \eqref{eq:second-derivative-F}.
\end{proof}

Formula \eqref{eq:second-derivative-F} is the intrinsic analogue of the familiar Euclidean expression
\[
F''(p)[u,u]+F'(p)a.
\]
The use of covariant acceleration makes the formula coordinate-independent.

Using Jacobi fields, we can find a particularly nice expression for the edge Hessian.

\begin{proposition}[Explicit description of the Riemannian Hessian]\label{prop:riemannian-edge-hessian}
	Given an intrinsic bar-and-joint framework $(G,p)$ on a Riemannian manifold $(\M,g)$ that has an admissible placement $p$, let $\gamma(t)=\Exp_{p_i}(t\cdot v_{ij})$ be the unique minimizing geodesic between $p_{i}$ and $p_j$ for an edge $ij\in E$ with $v_{ij}=\Log_{p_i}(p_j)$ so that $\gamma(0)=p_i$ and $\gamma(1)=p_j$. For $u,v\in T_p\M^n$, let $J_u$ be the Jacobi field along $\gamma$ satisfying $J_u(0)=u_i$ and $J_u(1)=u_j$. Then, for any edge $e\in E(G)$ the edge Hessian has the form
	\[(\mathcal{H}_p(u,v))_e=-\big\langle D_tJ_u(0),\,v_i\big\rangle_{p_i}+\big\langle D_tJ_u(1),\,v_j\big\rangle_{p_j}\]
	\begin{proof}
		First of all, the variation of the Riemannian logarithms is given by
		\[\nabla_u \Log_{p_i}(p_j)=D_tJ_u(0)\quad\text{ and }\quad \nabla_u \Log_{p_j}(p_i)=-D_tJ_u(1),\]
		where $\nabla_u \Log_{p_i}(p_j)$ means varying both points $p_i$ and $p_j$ according to $u_i$ and $u_j$ (cf.\ \cite{pennec}).
		Since the Riemannian gradient of $f_e$ satisfies $\mathrm{grad}\,f_e(p)=(0 \dots 0 ~ -\Log_{p_i}(p_j) ~ 0 \dots 0  ~ -\Log_{p_j}(p_i) ~ 0 \dots 0)$ (see Proposition \ref{prop:edge-row}), we can compute
		\begin{eqnarray*}
			(\mathcal{H}_p(u,v))_e &=& \mathrm{Hess}\, f_e(p)[u,v] \\
			&\overset{\text{Def. }\ref{def:hessian}}{=}& \big\langle \nabla_u\mathrm{grad}\,f_e(p),\,v\rangle_p\\
			&\overset{\text{Prop. }\ref{prop:edge-row}}{=}& -\big\langle \nabla_u \Log_{p_i}(p_j),\,v_i\big\rangle_{p_i} - \big\langle \nabla_u\Log_{p_j}(p_i),\,v_j\big\rangle_{p_j}\\
			&=& -\big\langle D_tJ_u(0),\,v_i\big\rangle_{p_i}+\big\langle D_tJ_u(1),\,v_j\big\rangle_{p_j}
		\end{eqnarray*}
	because the Levi-Civita connection is compatible with the Riemannian metric $g$, which concludes the proof.
	\end{proof}
\end{proposition}

\subsection{Second-order flexes and second-order rigidity}

\begin{definition}[Second-order flex]\label{def:second-order-flex}
	A pair $(u,a)\in T_p\mathcal{M}^n\times T_p\mathcal{M}^n$ is a \struc{second-order flex} of $(G,p)$ if
	\begin{align}
		R_\M(G,p)u&=0,\nonumber\\
		R_\M(G,p)a+\mathcal H_p(u,u)&=0.\label{eq:so2}
	\end{align}
	The vector $u$ is the first-order velocity and $a$ is a compatible covariant acceleration.
\end{definition}

Every genuine $C^2$ motion $q(t)$ with $F_G(q(t))=F_G(p)$ yields a second-order flex by taking $u=q'(0)$ and $a=\nabla_{q'}q'|_{t=0}$.

\begin{definition}[Second-order rigidity]\label{def:second-order-rigid}
	The framework $(G,p)$ is \struc{second-order rigid} if every second-order flex $(u,a)$ has $
	u\in \cK_p$.
	Equivalently, no nontrivial infinitesimal flex admits a compatible second-order acceleration.
\end{definition}

This definition sits strictly between first-order information and local rigidity: a vector may satisfy the linearized edge equations but fail the quadratic compatibility equations.

\subsection{Self-stresses and the exact second-order stress test}

Equip $\mathbb R^E$ with its standard inner product.  The adjoint of the rigidity operator is
\[
R_\M(G,p)^*:\mathbb R^E\longrightarrow T_p\mathcal{M}^n.
\]

\begin{definition}[Self-stress]\label{def:self-stress}
	A vector $\omega=(\omega_e)_{e\in E}\in\mathbb R^E$
	is a \struc{self-stress} at $p$ if $R_\M(G,p)^*\omega=0$
	or equivalently,
	\[
	\omega\in\ker R_\M(G,p)^*=(\im R_\M(G,p))^\perp.
	\]
	To a self-stress $\omega$, we associate the \struc{stress potential} $\Psi_\omega: \mathcal{C}_G\rightarrow \mathbb{R}$
	\[
	\Psi_\omega(q):=\sum_{e\in E}\omega_e f_e(q)
	\]
	which comes with the \struc{stress Hessian} $Q_\omega:T_p\M^n\times T_p\M^n\rightarrow \mathbb{R}$ given by
	\begin{equation*}
		Q_\omega(u,v)
		:=\ip{\omega}{\mathcal H_p(u,v)}_{\mathbb R^E}
		=\sum_{e\in E}\omega_e\,\Hess f_e(p)[u,v].
	\end{equation*}
\end{definition}

Because $\omega$ is a stress, it holds that
\begin{equation}\label{eq:stresspotential}
d(\Psi_\omega)_p=\omega^\top R_\M(G,p)=0,\quad \text{ so }\quad \Hess \Psi_\omega(p)=Q_\omega.
\end{equation}

\begin{lemma}[Trivial motions lie in the radical]\label{lem:radical}
	For every self-stress $\omega$,
	\[
	Q_\omega(k,v)=0
	\]
	for all $k\in \mathcal{K}_p$ and all $v\in T_p\mathcal{M}^n$.
	Consequently $Q_\omega$ descends to a well-defined quadratic form on the quotient $\ker R_\M(G,p)/\mathcal{K}_p.$
\end{lemma}

\begin{proof}
	The constraint map $F_G$ is invariant under the group action of $\Isom_\mathrm{id}(M,g)$, hence so is $\Psi_\omega$.  Let $\widetilde X$ be the vector field on $\mathcal{M}^n$ induced by an element of $\mathfrak{isom}(\M,g)$ (see \Cref{section:inf_flexes}).  Then $\widetilde X(\Psi_\omega)=0$ identically.  At the critical point $p$ of $\Psi_\omega$,
	\[
	\Hess_p\Psi_\omega(v,\widetilde X_p)
	=v\bigl(\widetilde X\Psi_\omega\bigr)
	-d(\Psi_\omega)_p(\nabla_v\widetilde X)
	=0.
	\]
	Since $\widetilde X_p\in \cK_p$, the claim follows.
\end{proof}

Since $\mathcal{K}_p$ and $\ker R_\M(G,p)$ are linear spaces with $\mathcal{K}_p\subset\ker R_\M(G,p)$ (see Proposition \ref{prop:killing-kernel}), Lemma \ref{lem:radical} guarantees that it suffices to consider any space of infinitesimal flexes $\mathcal{I}_p\subset \ker R_\M(G,p)$ which is complementary to the space of trivial infinitesimal flexes $\mathcal{K}_p$. In other words, it suffices to consider any subspace $\mathcal{I}_p\subset \ker R_\M(G,p)$ with the property that $\mathcal{I}_p\oplus \mathcal{K}_p=\ker R_\M(G,p)$ instead of having to deal with the nonlinear object $\ker R_\M(G,p)\,\setminus\,\mathcal{K}_p$.

\begin{proposition}\label{prop:stress-obstruction}
	If $(u,a)$ is a second-order flex, then $Q_\omega(u,u)=0$
	for every self-stress $\omega$.
\end{proposition}

\begin{proof}
	We pair Equation \eqref{eq:so2} with $\omega$ to obtain
	\[
	0
	=\big\langle\omega, \,R_\M(G,p)a+\mathcal H_p(u,u)\big\rangle
	=\big\langle\underbrace{R_\M(G,p)^*\omega}_{=0},\,a\big\rangle+Q_\omega(u,u)
	=Q_\omega(u,u).\qedhere
	\]
\end{proof}

The converse is just as important and is a consequence of the Fredholm Alternatives Theorem.

\begin{theorem}[Exact stress test]\label{thm:stress-test}
	Let $u\in\ker R_\M(G,p)$.  The following are equivalent:
	\begin{enumerate}[label=(\roman*)]
		\item there exists $a\in T_p\mathcal{M}^n$ such that $(u,a)$ is a second-order flex;
		\item $\mathcal H_p(u,u)\in\im R_\M(G,p)$;
		\item $Q_\omega(u,u)=0$ for every self-stress $\omega\in\ker R_\M(G,p)^*$.
	\end{enumerate}
\end{theorem}

\begin{proof}
	The equivalence of (i) and (ii) is given by Equation \eqref{eq:so2}.  In finite dimensions,
	\[
	\im R_\M(G,p)=\bigl(\ker R_\M(G,p)^*\bigr)^\perp.
	\]
	Therefore (ii) holds exactly when $\mathcal H_p(u,u)$ is orthogonal to every self-stress, which is (iii).
\end{proof}

\begin{corollary}[Stress characterization of second-order rigidity]\label{cor:stress-characterization}
	The framework is second-order rigid if and only if for every nonzero $
	u\in\ker R_\M(G,p)/ \cK_p$
	there exists a self-stress $\omega$ for which
	\[
	Q_\omega(u,u)\neq0.
	\]
\end{corollary}

This is the intrinsic Riemannian version of the classical second-order stress test.

\begin{remark}[Surjective rigidity operator]\label{rem:surjective}
	If $R_\M(G,p)$ is surjective, then $\ker R_\M(G,p)^*=\{0\}$.  Hence for every infinitesimal flex $u$ the equation
	\[
	R_\M(G,p)a=-\mathcal H_p(u,u)
	\]
	has a solution.  Therefore a framework with a surjective rigidity operator cannot be second-order rigid unless it is already infinitesimally rigid.  In an underbraced framework this is a useful diagnostic: the absence of self-stresses means that the second order criterion alone cannot obstruct any infinitesimal flexes.
\end{remark}

\subsection{Prestress stability}

Second-order rigidity allows the obstructing stress to depend on the infinitesimal flex.  A stronger and often easier-to-check condition asks for a single stress on the entire space of nontrivial infinitesimal flexes.

\begin{definition}[Riemannian prestress stability]\label{def:prestress}
	An admissible framework $(G,p)$ is \struc{prestress stable} if there exists a self-stress $\omega$ such that the induced quadratic form on $
	\ker R_\M(G,p)\,/ \,\cK_p$
	is positive definite.  
	
	Equivalently, $
	Q_\omega(u,u)>0$
	for every nonzero $u\in\ker R_\M(G,p)\,/\, \mathcal{K}_p$ after identifying vectors differing by a trivial infinitesimal motion.  According to Lemma~\ref{lem:radical}, this condition is independent of the representative of the quotient class.
\end{definition}

Let us now consider how prestress stability relates to second-order rigidity.
\begin{theorem}\label{thm:inf-prestress}
	Every infinitesimally rigid framework is prestress stable.
\end{theorem}
\begin{proof}
	For an infinitesimally rigid framework it holds that $\ker R_\mathcal{M}(G,p)=\mathcal{K}_p$. Therefore, $\ker R_\mathcal{M}(G,p)\,/\,\mathcal{K}_p=\{0\}$ and $Q_\omega(u,u)>0$ trivially holds for every nonzero $u\in \ker R_\mathcal{M}(G,p)\,/ \,\mathcal{K}_p$.
\end{proof}

\begin{theorem}\label{thm:prestress-so}
	Every prestress-stable framework is second-order rigid.
\end{theorem}

\begin{proof}
	If $(u,a)$ were a second-order flex with $u\notin \mathcal{K}_p$, Proposition~\ref{prop:stress-obstruction} would give $Q_\omega(u,u)=0$ for the stabilizing stress, contradicting positive definiteness on $\ker R_\M(G,p)\,/\,\mathcal{K}_p$.
\end{proof}

\subsection{Second-order rigidity implies local rigidity}

The following argument is the intrinsic analogue of the classical fact that a genuine finite flex produces a nontrivial second-order flex after taking the first nonzero asymptotic direction.  The proof becomes especially clear when using the language of stresses.

Recall that $\mathcal{O}_p$ denotes the trivial motions (see Equation \eqref{eqn:isometry-orbit}). Since $\M$ is connected, the diagonal group action of the isometry group on $\M^n$ is proper, so we can locally choose a smooth transverse slice $\mathcal{S}\subset\mathcal{M}^n$ through $p$ such that
\[
T_p\mathcal{M}^n=\cK_p\oplus T_p\mathcal{S}
\]
using the Slice Theorem \cite[Prop.\ 2.3.1]{palaisslice}. As in the first-order local-rigidity argument, every sufficiently close configuration can be moved by a nearby isometry into $\mathcal{S}$.

\begin{theorem}\label{thm:so-local}
 If $(G,p)$ is second-order rigid, then it is locally rigid.
\end{theorem}

\begin{proof}
	We prove the statement's contraposition: Suppose the framework is flexible. Since the identity component of $\textsc{Isom}(\M,g)$ acts properly on $\M^n$, there exists a local slice $\mathcal{S}\subset\M^n$ through $p$ for which $T_p\M^n=\mathcal{K}_p\oplus T_p\mathcal{S}$, so that $\mathcal{S}\cap \mathcal{O}_p=\{p\}$ and so that every sufficiently close realization to $p$ can be moved into $\mathcal{S}$ by an isometry. 
	After applying such isometries, there exists a sequence
	$(q_k)_{k\in\mathbb{N}}\subset\mathcal{S}$, for which $q_k\to p$ and $q_k\neq p$ for $k\in \mathbb{N}$ so that $F_G(q_k)=F_G(p)$.

	Choose a smooth chart $\chi:U\subset T_p\mathcal{S}\rightarrow \mathcal{S}$ with $\chi(0)=p$ and $d\chi_0=\mathrm{id}_{T_p\mathcal{S}}$ and write
	\[
	q_k=\chi(\xi_k)\quad\text{for some}\quad \xi_k\in T_p\mathcal{S}\quad\text{so that}
	\quad r_k:=\norm{\xi_k}\to0.
	\]
	After passing to a subsequence, we can assume that
	\[
	u_k:=\frac{\xi_k}{r_k}\longrightarrow u
	\]
	for some $u\in T_p\mathcal{S}$ with $\norm{u}=1$. Since the set of admissible placements $\mathcal{C}_G$ is an open set, $F_G$ is smooth in a neighborhood of $p$ and due to the convergence of $(q_k)_{k\in\mathbb{N}}$ it holds that $q_k\in \mathcal{C}_G$ for all sufficiently large $k\in \mathbb{N}$. The first-order Taylor expansion of the constraint map then yields
	\[F_G(\chi(\xi))=F_G(p)+R_\M(G,p)\xi+o(\|\xi\|).\]
	When applying this equation to the sequence $(\xi_k)_{k\in\mathbb{N}}$, we obtain
	\[0=R_\mathcal{M}(G,p)\xi_k+o(r_k).\]
	Dividing by $r_k$ and passing to the limit yields $
	R_\M(G,p)u=0$.
	Since $u\in T_p\mathcal{S}$ and $T_p\mathcal{S}\cap \mathcal{K}_p=\{0\}$, $u$ therefore is a nontrivial infinitesimal flex.
	
	Now let $\omega$ be any self-stress and consider the scalar stress potential $\Psi_\omega(q)=\langle \omega,\,F_G(q)\rangle$, which has the property that $d(\Psi_\omega)_p=0$ and $\mathrm{Hess}\,\Psi_\omega(p)=Q_\omega$ (see Equation \eqref{eq:stresspotential}). Therefore, the second derivative of $\Psi_\omega \circ \chi$ at $0$ agrees with the Riemannian Hessian:
	\[d^2(\Psi_\omega\circ \chi)_0[u,u]=\mathrm{Hess}\,\Psi_\omega(p) [u,u] = Q_\omega(u,u).\]
	Putting all of these observations together, we find that the second-order Taylor expansion of $\Psi\circ\chi$ at $\chi(0)=p$ has the form
	\[\Psi_\omega(\chi(\xi))=\Psi_\omega(p)+\frac{1}{2}Q_\omega(\xi,\xi)+o(\|\xi\|^2).\]
	Since $F_G(q_k)=F_G(p)$, it also holds that $\Psi_\omega(q_k)=\Psi_\omega(p)$. Therefore,
	\[
	0
	=\frac12 {r_k}^2 Q_\omega(u_k,u_k)+o({r_k}^2).
	\]
	After dividing by ${r_k}^2$ and passing to the limit, it holds that $Q_\omega(u,u)=0$
	for every self-stress $\omega$.  By Theorem~\ref{thm:stress-test}, there exists $a\in T_p\mathcal{M}^n$ such that $(u,a)$ is a second-order flex.  Since $u\notin \cK_p$, this implies that $(G,p)$ is second-order flexible.
\end{proof}

In this proof, an important step is to directly pair a self-stress $\omega$ with the second-order Taylor expansion. Otherwise, one generally has extra terms in arbitrary coordinates known as the Christoffel symbols associated with the Riemannian metric. However, these terms are multiplied by $d(\Psi_\omega)_p$ and can therefore safely be ignored after introducing the self-stress.

\begin{corollary}\label{cor:prestress-local}
	Prestress stability implies local rigidity.
\end{corollary}

Thus, the hierarchy relevant to the intrinsic Riemannian setting therefore is
\[
\boxed{
		\text{infinitesimally rigid}
	\Longrightarrow
	\text{prestress stable}
	\Longrightarrow
	\text{second-order rigid}
	\Longrightarrow
	\text{locally rigid}.
}
\]

Already in the Euclidean setting, none of these implications are reversible (cf.\ \cite{connelly1996}). Recent work has generalized these observations, particularly to frameworks constrained to a certain smooth manifold \cite{polytoperigidity2}. This makes it likely that none of these implications are reversible for intrinsic bar-and-joint frameworks on general Riemannian manifolds.

\begin{remark}
	Deciding second-order rigidity or prestress stability makes the computation of the bilinear form $Q_\omega$ necessary which, in turn, demands the computation of all edge Hessians. Typically, this is significantly harder than in Euclidean space, as it requires the evaluation of covariant derivatives at the realization $p$. Analytically finding the Riemannian Hessian is often out of reach. However, we can use numerical differentiation schemes and ODE-solvers to approximate the coefficients of $Q_\omega$. If the matrix associated with this bilinear form is positive or negative definite and if the corresponding eigenvalues are bounded away from $0$, this approach provides reasonable numerical evidence.
\end{remark}

Let us now return to Example \ref{ex:hyperboloid} and use the tools developed in this section to show that the corresponding intrinsic framework is prestress stable. 

\begin{example}\label{ex:second-order-comp}
	Again consider the 3-cycle embedded at height $z=0$ on the one-sheeted hyperboloid $\mathcal{H}$ with the induced Euclidean metric. We show that this framework is prestress stable using different techniques than before. Using the parametrization of $\mathcal{H}$ in terms of $\theta\in[0,2\pi)$ and $z\in \mathbb{R}$ from before with the induced local frame $(\partial_\theta,\partial_z)$. Therefore, every infinitesimal flex has the form $u_i=a_i\partial_\theta+b_i\partial_z$. By the first variation formula for geodesic length, the bar constraints imply that $a_2-a_1=0$, $a_3-a_2=0$ and $a_1-a_3=0$, so that $a_1=a_2=a_3$. Since $\mathcal{H}$ is a surface of revolution, $a_i\neq 0$ and $b_i=0$ corresponds to a trivial infinitesimal rotation, so we can set $a_1=a_2=a_3=0$ to factor out the trivial infinitesimal flexes. In this local frame, the nontrivial infinitesimal flexes are therefore generated by linear combinations of the transverse motions $u_i=b_i(\partial_z)\lvert_{p_i}$ with $u_j=0$ for $j\neq i$. We set $u=(b_1(\partial_z)\lvert_{p_1},\,b_2(\partial_z)\lvert_{p_2},\,b_3(\partial_z)\lvert_{p_3})$ for $b_i\in\mathbb{R}$. This already generates a three-dimensional space of infinitesimal motions. By the Maxwell condition from Corollary \ref{cor:maxwell}, there is one more infinitesimal flex than expected. Since at any vertex, the two geodesics directions are opposite and have equal length, there exists an equilibrium stress of the form $\omega=(1,1,1)$. 
	
	Let us now consider the edge $ij$ of length $\ell>0$ with normal endpoints velocities $b_i$ and $b_j$ as before. Since the Gaussian curvature of the one-sheeted hyperboloid is $K=-1$ in the points on the geodesic $\gamma(\theta)=(\cos(\theta),\sin(\theta),0)$ at height $z=0$ with unit-speed parametrization, its normal Jacobi field $J(s)=j(s)\partial_z$ satisfies
	\[j''(s)-j(s)=0,\quad j(0)=b_i\quad\text{ and }\quad j(\ell)=b_j\]
	We can analytically solve this differential equation and find that
	\[j(s)=\frac{\sinh(\ell-s)}{\sinh(\ell)}b_i+\frac{\sinh(s)}{\sinh(\ell)}b_j.\]
	Therefore, the edge Hessians of the constraint map's components are given by (cf.\ \cite[Thm.\ 4.1.1]{JostRieGeo})
	\begin{eqnarray*}
	\left(\mathcal{H}_p(u,u)\right)_{ij}&=&\ell\cdot\int_{0}^\ell \left(j'(s)^2+j(s)^2\right) \,ds\\&=&\ell\cdot \frac{\cosh(\ell)\cdot({b_i}^2+{b_j}^2)-2b_ib_j}{\sinh(\ell)}\\&=&\ell\cdot \frac{(b_i-b_j)^2+(\cosh(\ell)-1)\cdot({b_i}^2+{b_j}^2)}{\sinh(\ell)}
	\end{eqnarray*}
	The integral comes from so-called \struc{variation fields} of the geodesic $\gamma$ (see also \cite[p.\ 153f.]{Lee2018}). Since $\ell>0$, it holds that $\sinh(\ell)>0$ and $\cosh(\ell)>1$. This implies that $\left(\mathcal{H}_p(u,u)\right)_{ij}>0$ as long as $b_i\neq 0\neq b_j$. Therefore, 
	\[Q_\omega(u,u)~=~\sum_{ij\in E}\omega_{ij}\cdot\left(\mathcal{H}_p(u,u)\right)_{ij}~=~\sum_{ij\in E}\left(\mathcal{H}_p(u,u)\right)_{ij}.\]
	Every summand is nonnegative and the sum is zero if and only if $b_1=b_2=b_3=0$ by a previous argument. However, this exactly represents the condition that $u$ is trivial. Thus, $Q_\omega(u,u)>0$ for every nontrivial infinitesimal flex $u$ and the bar-joint framework is prestress stable according to Definition \ref{def:prestress}.
\end{example}

\section{Specialization to known examples}\label{section:specialization-examples}
Let us now consider three established settings in which bar-and-joint frameworks have been considered, as well as one new setting arising from the theory developed in this article. The established settings that we consider are Euclidean space, the unit sphere and the two-sheeted hyperboloid, while the new result is on the statistical probability simplex. 

\subsection{Recovery of the Euclidean rigidity matrix}

Take $\M=\R^d$ with the standard Euclidean metric. Then,
\[
 \Log_{p_i}(p_j)=p_j-p_i,
 \qquad
 \Log_{p_j}(p_i)=p_i-p_j.
\]
The constraint map $F_G(p)$ is then given by the squared edge lengths, so it holds that $F_G(p)=\left(\frac{1}{2}\|p_i-p_j\|^2\right)_{ij\in E}$. The edge Equation \eqref{eq:edge-equation} therefore becomes
\[
 0\,=\,(p_j-p_i)\cdot u_i+(p_i-p_j)\cdot u_j\,=\, (p_i-p_j)\cdot(u_j-u_i).
\]
Up to multiplication by $(-1)$, this is a row of the classical Euclidean rigidity matrix $R_G(p)$.  Thus, Definition~\ref{def:rigidity-operator} is a direct intrinsic extension of the Euclidean infinitesimal rigidity theory.

Let us now fix a curve $q(t)=p+t\cdot u + \frac{1}{2}t^2\cdot a + o(t^2)$ with $q'(0)=u$ and $q''(0)=a$. The second-order condition \eqref{eq:so2} then becomes 
\[\frac{\partial^2}{\partial t^2}F_G(q(t))\big\lvert_{t=0}=R_G(p)\cdot a + R_G(u)\cdot u=0,\]
which recovers the standard second-order rigidity criterion (e.g.\ \cite{connelly1996}).

\subsection{The unit sphere}

Let $\M=\Sph^d\subset\R^{d+1}$ with the induced metric from the ambient Euclidean space.  For $p,q\in\Sph^d$ with $q\neq -p$, set
\[
 \theta=d_{\Sph^d}(p,q)=\arccos(\langle p,\, q\rangle ),
 \qquad 0\le\theta<\pi.
\]
For $p\neq  q$ and $\theta=\arccos(\langle p,\,q\rangle)$ the Riemannian logarithm is given by
\begin{equation*}\label{eq:sphere-log}
 \Log_p(q)=\frac{\theta}{\sin\theta}
 \bigl(q-(\cos\theta)p\bigr).
\end{equation*}
The expression extends continuously to $0$ at $q=p$.
Accordingly, for an edge $ij$ with
\[
 \theta_{ij}=\arccos(\langle p_i,\, p_j\rangle)<\pi,
\]
which is exactly the case when $p_i\neq -p_j$, the intrinsic spherical infinitesimal constraint is
\begin{align}\label{eq:sphere-row}
 0={}&
 \frac{\theta_{ij}}{\sin\theta_{ij}}\cdot \left(
 \bigl\langle p_j-(\cos\theta_{ij})p_i,\,u_i\bigr\rangle \,+ \,
 \bigl\langle p_i-(\cos\theta_{ij})p_j,\,u_j\bigr\rangle\right),
\end{align}
where $u_i\perp p_i$ and $u_j\perp p_j$. Since the scalar factor $\theta_{ij}/\sin\theta_{ij}$ is nonzero for $0<\theta_{ij}<\pi$, it may be ignored without changing the infinitesimal flex space. Therefore, \eqref{eq:sphere-row} can be rewritten as 
\begin{align}
	0={}&
	\bigl\langle p_j-(\cos\theta_{ij})p_i,\,u_i\bigr\rangle \,+ \,
	\bigl\langle p_i-(\cos\theta_{ij})p_j,\,u_j\bigr\rangle,
	\nonumber\\
		={}&\bigl\langle p_j,\,u_i\rangle-\langle p_i,p_j\rangle \cdot \underbrace{\langle p_i,\,u_i\bigr\rangle}_{=0} \,+ \,
	\bigl\langle p_i,\, u_j\rangle -\langle p_i,p_j\rangle \cdot \underbrace{\langle p_j,\,u_j\bigr\rangle}_{=0},
	\nonumber\\
	={}&\bigl\langle p_j,\,u_i\bigr\rangle \,+ \,
	\bigl\langle p_i,\,u_j\bigr\rangle,
	\label{eq:spherical_condition}
\end{align}
which recovers the spherical rigidity matrix (cf.\ \cite{sphericalrigiditymatrix}). 
The identity component of the isometry group is $\mathrm{SO}(d+1)$, which has dimension
 $\frac{d(d+1)}2$.
For a placement with discrete stabilizer, an infinitesimally rigid framework on $\Sph^d$ therefore satisfies the necessary count
\begin{equation}\label{eq:sphere-count}
 |E|\ge dn-\frac{d(d+1)}2,\nonumber
\end{equation}
which is consistent with the existing combinatorial classification on the $d$-sphere (cf.\ \cite{NixonOwenPower2012,SaliolaWhiteley2007}).

\subsubsection{The regular cube}\label{section:regular-cube}
Let us now consider the cube graph together with the regular embedding on the sphere, for which the vertices are given by $\frac{1}{\sqrt{3}}(\pm 1, \pm 1,\pm 1)$ (see \Cref{fig:regcube}). For the explicit computations on the rigidity operator, see Appendix Section \ref{section:regcube_computations}. The cube framework provides a blueprint for checking the second-order rigidity results from \Cref{section:second-order-theory} on a concrete example: When realized on the sphere, the cube is not first-order rigid, as it has 12 edges, leading to the deficient Maxwell condition $|E|=12<13=2\cdot 8 - \frac{2\cdot 3}{2}$, which means that the framework is underconstrained. Indeed, a simple computation reveals that it possesses two nontrivial infinitesimal flexes and one self-stress $\omega=(1,\dots,1)\in \mathbb{R}^{12}$, which is known as the \struc{Wachspress stress} \cite{wintercpf}  (see Appendix Section \ref{section:cube:rig_matrix}). 

\begin{figure}[h!]
	\centering
	\includegraphics[width=0.35\linewidth]{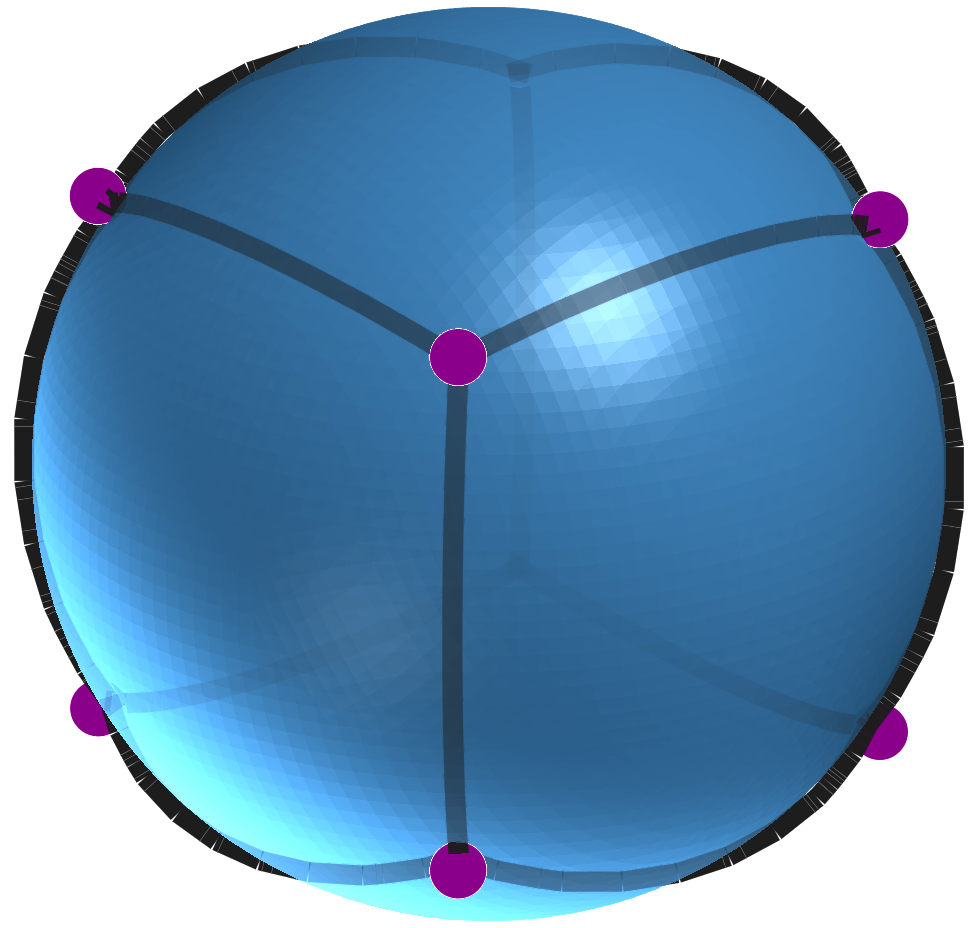}
	\caption{The regular cube embedded on the sphere, with geodesic arcs as edges.}
	\label{fig:regcube}
\end{figure}

At a vertex $p_i\in \mathbb{S}^2$, the nontrivial infinitesimal flexes are $u^{(1)}$ and $u^{(2)}$, which can be obtained from a basis of the linear subspace generated from the three vectors $v^{(1)}=p_{i,1}\cdot \left(e_x \times p_i\right)$, $v^{(2)}=p_{i,2}\cdot \left(e_y \times p_i\right)$ and $v^{(3)}= p_{i,3}\cdot \left(e_z \times p_i\right)$ with the standard unit basis $(e_x,e_y,e_z)$ of $\mathbb{R}^3$. Geometrically, these vectors represent a twist of two opposite faces. Since $\pm e_x, \pm e_y$ and $\pm e_z$ are no vertices of the cube, we can simply choose $u^{(1)}= v^{(1)}$ and $u^{(2)}= v^{(2)}$ as a basis for the space of nontrivial infinitesimal flexes  (cf.\ Appendix Section \ref{section:cube_inf_flexes}). 

Consider the curve $c(t)=\langle p_i(t),\,p_j(t)\rangle$ with geodesics $p_k(t)$ on $\mathbb{S}^2$ that satisfy $p_k(0)=p_k$ and ${p_k}'(0)=u_k$ for an infinitesimal flex $u$ and $k\in\{i,j\}$. The geodesic differential equation for a curve $\gamma$ on $\mathbb{S}^2$ becomes $\|\dot{\gamma}^2\|+\langle \ddot{\gamma},\,\gamma\rangle=0$. In addition, the vector field $\gamma''(t)$ must be orthogonal to $T_{\gamma(t)}\mathbb{S}^2$ for any geodesic (see Definition \ref{def:geodesics}) and for every $t$, which turns the geodesic differential equation into $\gamma''(t)=-\|\gamma'(t)\|^2\cdot \gamma(t)$. Therefore, for our curves $p_k$ this expression becomes ${p_k}''(0)=-\|u_k\|^2 \cdot p_k$ at $t=0$.
The first two derivatives of the curve $c$ then are  
\[c'(0)=\langle u_i,\,p_j\rangle +\langle p_i,\,u_j\rangle=0 \quad\text{ and }\quad c''(0)=2\langle u_i,u_j\rangle - \langle p_i,p_j\rangle \cdot (\|u_i\|^2+\|u_j\|^2).\] 
If we now choose $h(c)=\frac{1}{2}\arccos(c)^2$, then $h(c(t))=\frac{1}{2}d_{\mathbb{S}^2}(p_i(t),p_j(t))^2$ with the derivative $h'(c)=-\frac{\arccos(c)}{\sin(\arccos(c))}$. At $t=0$, this function satisfies the equation 
\[\frac{\partial^2}{\partial t^2}h(c(t))\bigg\lvert_{t=0}=h''(c(0))\cdot \underbrace{\left(c'(0)\right)^2}_{=0}+h'(c(0))\cdot c''(0)=h'(c(0))\cdot c''(0).\] 
Consequently, the spherical edge Hessians have the form
\begin{eqnarray}\label{eq:spherical_edge_hessian}
	H_{ij}(u,u)=-\frac{\arccos(\langle p_i,\, p_j\rangle)}{\sin(\arccos(\langle p_i,\, p_j\rangle))}\cdot \left(2\langle u_i,\,u_j\rangle- \langle p_i,p_j\rangle\cdot \left(\|u_i\|^2+\|u_j\|^2\right)\right)
\end{eqnarray}
In Appendix Section \ref{section:cube:edge_hessian}, we use $H_{ij}$ to derive an expression for the stress Hessian $Q_\omega$ in terms of the general nontrivial flex $u=\lambda_1\cdot u^{(1)}+\lambda_2\cdot u^{(2)}$ with variables $\lambda_1$, $\lambda_2$:
\[Q_\omega(u,u)~=~{16}\cdot \frac{\arccos\left(\frac{1}{3}\right)}{\sqrt{2}\cdot 3} \cdot \left[\left(\lambda_1-\frac{\lambda_2}{2}\right)^2+\frac{3}{4}{\lambda_2}^2\right].\]
Therefore, $Q_\omega(u,u)>0$ for all $(\lambda_1,\lambda_2)\in \mathbb{R}^2\setminus\{0\}$ and the cube is prestress stable according to Definition \ref{def:prestress}. Consequently, this framework is rigid by Corollary \ref{cor:prestress-local}, but not infinitesimally rigid.

\subsection{The two-sheeted hyperboloid}

Consider the two-sheeted hyperboloid $\mathcal{H}$ in $\mathbb{R}^3$ cut out by the implicit equation $x^2+y^2-z^2=-1$ together with the restriction of the Minkowski metric $g_{2,1}=dx^2+dy^2-dz^2$, which is also known as the standard hyperbolic Riemannian metric on $\mathcal{H}$. This Riemannian manifold is depicted in \Cref{fig:hyperboloid}. This metric produces a constant Gaussian curvature on $\mathcal{H}$. We deliberately use the two-sheeted hyperboloid instead of the one-sheeted version which is more common in the literature, as the latter is not a Riemannian manifold when equipped with the Minkowski metric.

\begin{figure}[h!]
	\centering
	\includegraphics[width=0.35\linewidth]{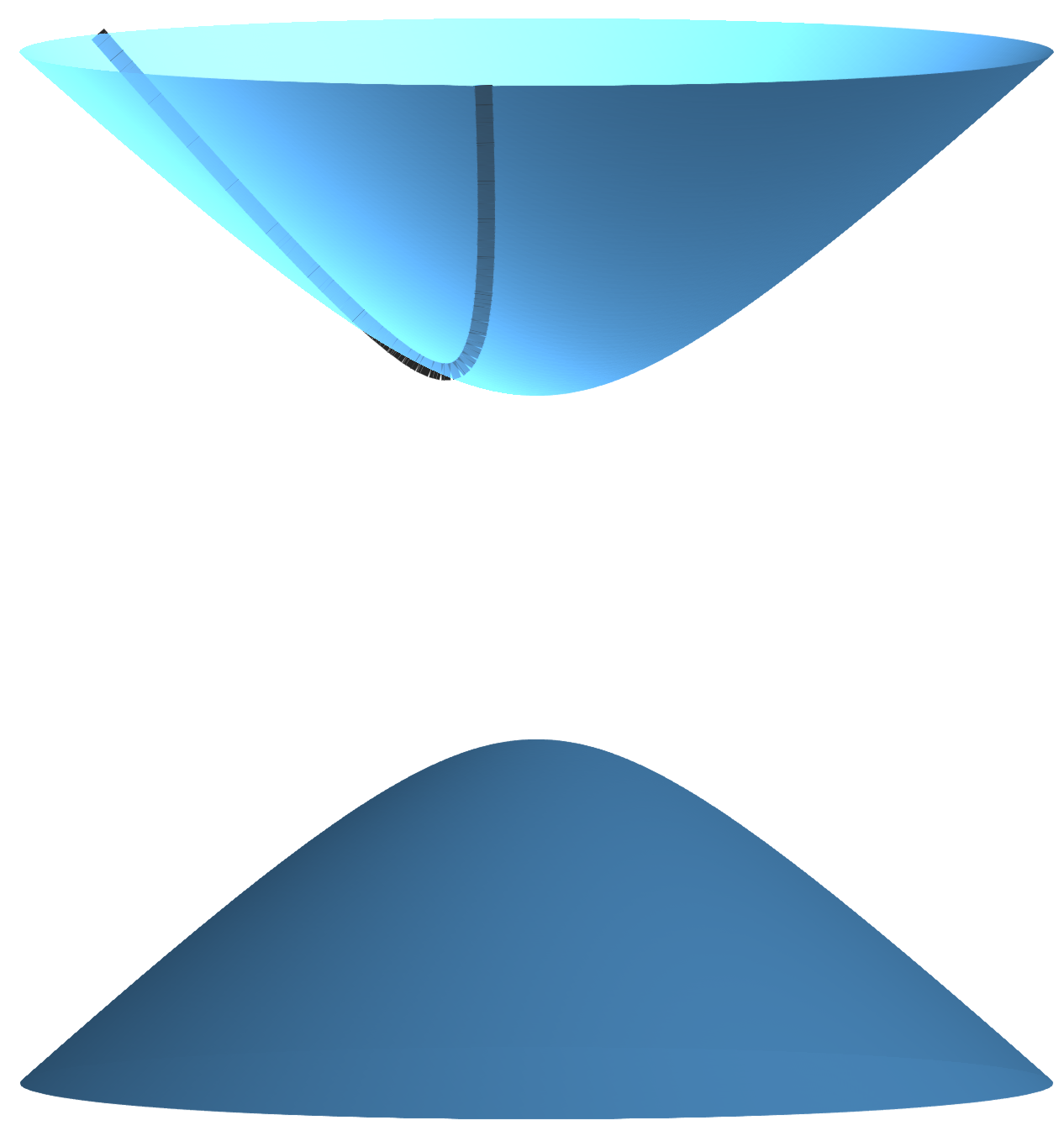}
	\hspace*{20mm}
	\includegraphics[width=0.32\linewidth]{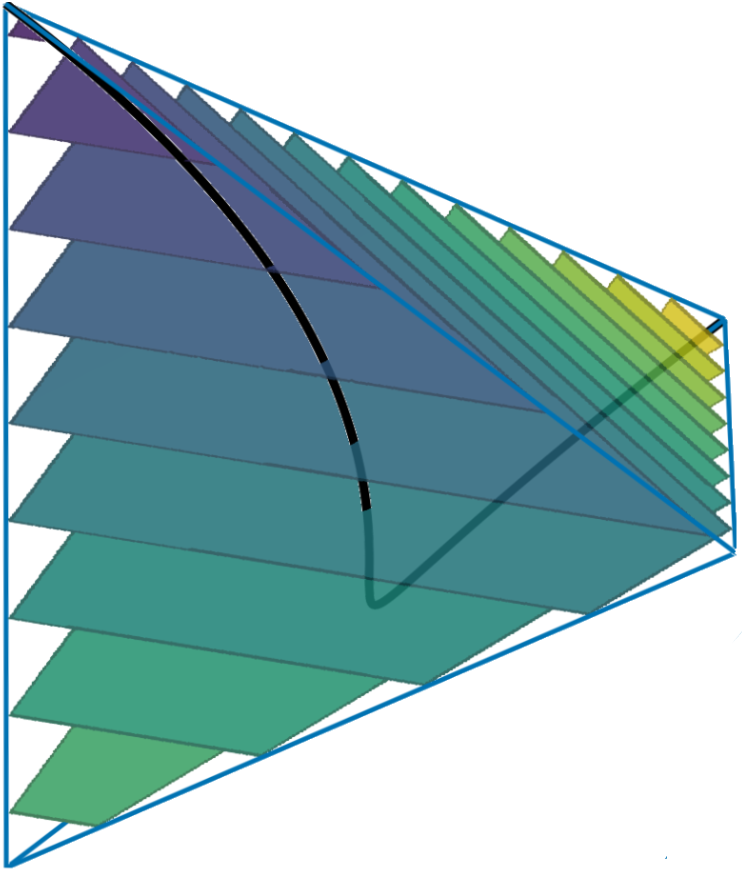}
	\caption{\textbf{(l.)} The two-sheeted hyperboloid with a geodesic. \textbf{(r.)} Normal spaces of the twisted cubic $\rho\mapsto (\rho^3,\,3\rho^2(1-\rho),\,3\rho(1-\rho)^2,\,(1-\rho)^3))$ for $\rho\in[0,1]$ with its normal spaces inside a projection of the probability simplex $\Delta_3$ to $\mathbb{R}^3$.}
	\label{fig:hyperboloid}
\end{figure}

On the hyperboloid, the geodesic distance between the points $p$ and $q$ that lie on the same sheet is given by $d_{2,1}(p,q)=\mathrm{arccosh}(-\langle p,q\rangle_{2,1})$ and the geodesics are given by intersections of either of the hyperboloid's sheets with 2-dimensional linear subspaces through the origin \cite[Prop.\ 5.28]{Lee2018}. Therefore, it holds that 
\[F_G(p)=\left(\frac{1}{2} d_{2,1}(p_i,p_j)^2\right)_{ij\in E}=\left(\frac{1}{2}\mathrm{arccosh}(-\langle p_i,p_j\rangle_{2,1})^2\right)_{ij\in E}.\]
Differentiating this expression provides the identity
\[\frac{\mathrm{arccosh}(-\langle p_i,p_j\rangle_{2,1})}{\sqrt{(\langle p_i,p_j\rangle_{2,1})^2-1}}\cdot \left(\langle u_i,p_j\rangle_{2,1}+\langle p_i,u_j\rangle_{2,1}\right)=0,\]
which can be used for finding infinitesimal flexes $u$. This expression simplifies to the following: $\langle u_i,p_j\rangle_{2,1}+\langle p_i,u_j\rangle_{2,1}=0$ in case $\langle p_i,p_j\rangle_{2,1}\neq -1$, which is always true as long as $p_i \neq p_j$ lie on the same sheet. This observation is consistent with existing results from the literature (cf.\ \cite{SaliolaWhiteley2007}).

\subsection{The statistical probability simplex}

The \struc{open statistical probability simplex} $\Delta_{d-1}\subset \mathbb{R}^d_{>0}$ is defined by
\[\Delta_{d-1}=\left\{(x_1,\dots,x_d)\in \mathbb{R}^d_{>0}~\Bigg\lvert~\sum_{i=1}^d x_i=1\right\}.\]
It can be turned into a Riemannian manifold when considering the \struc{Fisher metric} $g=\sum_{i=1}^d \frac{1}{x_i} (dx_i)^2$ on $\Delta_{d-1}$ (cf.\ \cite{informationgeometry}). If we consider a framework $(G,p)$ on the statistical probability simplex, the corresponding Fisher distance on the edges is given by
\[d_\Delta(p,q)=2\arccos\left(\sum_{i=1}^d \sqrt{p_i\cdot q_i}\right).\]
If we now differentiate the squared distance function $F_G(p)=\left(\frac{1}{2} d_\Delta (p_i,p_j)^2\right)_{ij \in E}$, we obtain the rigidity operator 
\begin{eqnarray}
	\label{eqn:fisher-rigidity-operator}
	\hspace*{12mm}\left(R_{\Delta_{d-1}}(G,p)u\right)_{ij} \,=\, -\frac{2\arccos(\sum_{k=1}^d\sqrt{p_{i,k} \cdot p_{j,k}})}{\sqrt{1-\left(\sum_{k=1}^d{\sqrt{p_{i,k}\cdot p_{j,k}}}\right)^2}}\cdot \left( \sum_{k=1}^d \sqrt{\frac{p_{j,k}}{p_{i,k}}}u_{i,k}+\sum_{k=1}^d\sqrt{\frac{p_{i,k}}{p_{j,k}}}u_{j,k}\right)
\end{eqnarray}
for an infinitesimal flex $u\in \mathbb{R}^{nd}$ that additionally satisfies the tangency condition $0=\sum_{k=1}^d u_{i,k}$ at any fixed vertex $i\in V(G)$, where $u_{i,k}$ denotes the $k$-th component of the vector $u_i$. Since $u$ additionally is supposed to lie in the kernel of the rigidity operator, we can set the above equation to $0$ and ignore the factor. 

Now, we choose square root coordinates of the simplex: For $p=(p_1,\dots,p_n)\in \Delta_{d-1}$ with $\sum_{k=1}^d p_{i,k}=1$ we choose $x_i=2\sqrt{p_i}$ so that $\sum_{i=1}^d {x_i}^2=4\sum_{k=1}^d p_{i,k}=4$ and $x_i\in 2\cdot \mathbb{S}^{d-1}$. If now $p(t)$ is a continuously differentiable curve in the simplex with $u_{i,k}=\dot{p}_{i,k}$, we can differentiate in the new coordinates to obtain $\dot{x}_{i,k}=\frac{u_{i,k}}{\sqrt{p}_{i,k}}$. Thus, the tangency condition becomes $x_i^\top \dot{x}_i=0$ and Equation \eqref{eqn:fisher-rigidity-operator} becomes $x_i^\top\dot{x}_j+x_j^\top \dot{x}_i=0$ up to multiplication with scalar factors.

\begin{remark}
	The condition $x_i^\top\dot{x}_j+x_j^\top \dot{x}_i=0$ for infinitesimal flexes looks identical to the spherical condition \ref{eq:spherical_condition}. However, contrary to the sphere, the isometry group of $\Delta_{d-1}$ is discrete. Therefore, for $d\geq 2$ every finite intrinsic framework on the open probability simplex is locally flexible, as in square root coordinates sufficiently small infinitesimal rotations on the positive orthant of a sphere preserve the geodesic edge lengths. Since the global isometries of $\Delta_{d-1}$ are coordinate permutations, these infinitesimal rotations are nontrivial and locally extend to smooth motions. Consequently, while the rigidity operator presented in this section is a useful differential formula, it does not produce locally rigid frameworks after factoring out the Lie group of isometries. For that specific reason, Remark \ref{rem:local_isometry_group} discusses using local isometries instead, which come with several major theoretical drawbacks. 
\end{remark}

\section{Concluding Remarks}
In this article, we develop an intrinsic rigidity theory on Riemannian manifolds. This encompasses an analogue of the local rigidity of an intrinsic bar-and-joint framework, as well as an infinitesimal and second-order rigidity theory for these objects. 

For an admissible intrinsic framework $(G,p)$ on a Riemannian manifold $(\M,g)$, the constructions from this article may be summarized as follows:
\begin{align*}
	F_G(p)_{ij}=\frac12\ell_g(p_i,p_j)^2,&\quad
	R_\M(G,p)=d(F_G)_p,\\
	\cK_p=\{(A(p_1),\dots,A(p_n))\,:\,A\in\mathfrak{isom}(\M,g)\},&\quad \mathcal{O}_p=\textsc{Isom}_{\mathrm{id}}(\M,g)\cdot p,\\[4mm]
	\hline\\[-1mm]
	(G,p)\text{ infinitesimally rigid}
	&\Longleftrightarrow
	\ker R_\M(G,p) =\cK_p,\\
	(G,p) \text{ rigid} &\Longleftrightarrow \exists\, U(p)\subset \M^n:~F_G^{-1}(F_G(p))\cap U(p)\subset \mathcal{O}_p\\
	(G,p) \text{ flexible + additional assumptions} & \Longrightarrow \exists \text{ nontrivial motion},\\
	\forall\, \text{nontrivial flex }u~\exists\,\text{self-stress } \omega:~Q_\omega (u,u)\neq 0&\Longleftrightarrow (G,p) \text{ second-order rigid},\\
	\exists\,\text{self-stress } \omega~\forall\, \text{nontrivial flex }u:~Q_\omega (u,u)\neq 0&\Longleftrightarrow (G,p) \text{ prestress stable}.\\[4mm]
	\hline\\[-1mm]
	\text{infinitesimally rigid}
	\Longrightarrow
	\text{prestress stable}
	&\Longrightarrow
	\text{second-order rigid}
	\Longrightarrow
	\text{locally rigid}.
\end{align*}
This formulation is invariant under change of coordinates and, in the Euclidean case, reduces exactly to the classical rigidity theory for bar-and-joint frameworks. In this article, we also separate the genuine Riemannian input, i.e.\ the variation of squared geodesic distance, from the graph-theoretic structure of the rigidity problem.

Since for general manifolds, the Lie group of isometries is usually low-dimensional, the infinitesimal theory is more often than not insufficient. This makes the development of the second-order rigidity theory particularly valuable and we are able to apply this theory to prove the rigidity of multiple examples that are not infinitesimally rigid, yet second-order rigid (see Examples \ref{ex:hyperboloid}, \ref{ex:second-order-comp} and \Cref{section:regular-cube}).

\begin{remark}
Under the square-root transformation, the Fisher probability simplex is isometric to the positive orthant of a round sphere. Small spherical rotations of a finite configuration preserve its intrinsic pairwise distances without necessarily extending to global isometries of the simplex. Consequently, even complete frameworks may be locally flexible under Definition \ref{def:local_rigidity}, further illustrating the distinction between global and local isometries discussed in Remark \ref{rem:local_isometry_group}.
\end{remark}

\subsection{Degenerate metrics}

The preceding definitions require $g$ to be a Riemannian metric, i.e.\ a positive-definite bilinear form on the manifold's tangent bundle.  Suppose instead that a continuously differentiable symmetric tensor $\widetilde g$ is positive definite away from a discrete set and vanishes at these points.  In that case, $\widetilde g$ is a degenerate or singular metric rather than a Riemannian metric.  Several ingredients used above may fail at the degenerate points:
\begin{enumerate}[label=(\roman*)]
 \item The pointwise norm on the tangent space may vanish on nonzero vectors, which may lead to ``invisible'' tangent directions.
 \item The Riemannian logarithm need not be defined by the usual inverse-function argument for the exponential map. It may still be possible to define the Riemannian logarithm via another approach though (see e.g.\ \cite[Def.\ 10.20]{boumal2020intromanifolds} or \Cref{section:generalization}).
 \item The induced distance does not need to be smooth, or even locally comparable to a smooth coordinate distance. Instead, distances may scale like a nonlinear power of the coordinate difference or distinct points may even have zero pseudodistance: For a smooth Riemannian metric, the squared distance $\ell_g(p,q)^2$ behaves quadratically near $p=q$. If the metric degenerates, that quadratic term can disappear in singular tangent directions, so that  $\ell_g(p,q)^2$ behaves like $|x|^{2k}$ for some $k>1$ and the differential $d(F_G)_p$ consequently vanishes to higher order than 1, meaning that the rigidity operator's kernel is not a meaningful object to consider anymore. 
\end{enumerate}
Consequently, the rigidity operator should not be used without further analysis in the degenerate setting.  A natural problem is to develop a higher-order or nonsmooth rigidity theory adapted to such singular metrics.  In particular, if a conformal factor vanishes at the framework vertices, first-order edge constraints may lose rank even when higher-order distance terms are nontrivial.

In this context, an interesting object of study could be \struc{pseudo-Riemannian manifolds}, solely requiring that the corresponding bilinear form is non-degenerate \cite[p.\ 42]{Lee2018}. Several properties of Riemannian manifolds persist: There still exists a unique Levi-Civita connection \cite[p.\ 61]{semiriemanniangeometry}, which enables concepts such as geodesics (cf.\ Definition \ref{def:geodesics}).

Alternatively, one could slightly perturb the metric by setting $\widetilde g_\varepsilon=\widetilde g+\varepsilon\cdot h$ for any Riemannian metric $h$ on $\M$ to study the resulting Riemannian metric $\widetilde g_\varepsilon$. It is likely that this metric behaves similarly to $\widetilde g$ and allows the use of the tools that are developed throughout this article.

\subsection{Generalization of admissible placements}\label{section:generalization}

The constructions from this article also require that the geodesic arcs connecting two vertices are length-minimizing. This assumption has the consequence that the geodesic between two admissible points on the manifold is uniquely determined and that the rigidity operator is completely determined by the vertex positions alone. Ultimately, this lets us define the Riemannian logarithm as the inverse function of the exponential map, which is a convenient tool to work with, as we demonstrate throughout. 

But what if we relaxed the assumption that a placement $p\in \mathcal{M}^n$ is admissible for the graph $G$? In that case, we would allow pairs of points to lie outside of the region $\mathcal{U}_\mathrm{Log}$ for which the exponential map is not necessarily injective, so the Riemannian logarithm is not well-defined. In such points, the Riemannian distance may fail to be differentiable. Consider for instance a great circle on the sphere that is depicted in \Cref{fig:greatcircle}. There are two distinct geodesic arcs between two distinct points on the great circle. However, this 2-gon may still be a reasonable framework to consider and clearly, it is rigid when the points are not antipodal. 

\begin{figure}[h!]
	\includegraphics[width=0.335\linewidth]{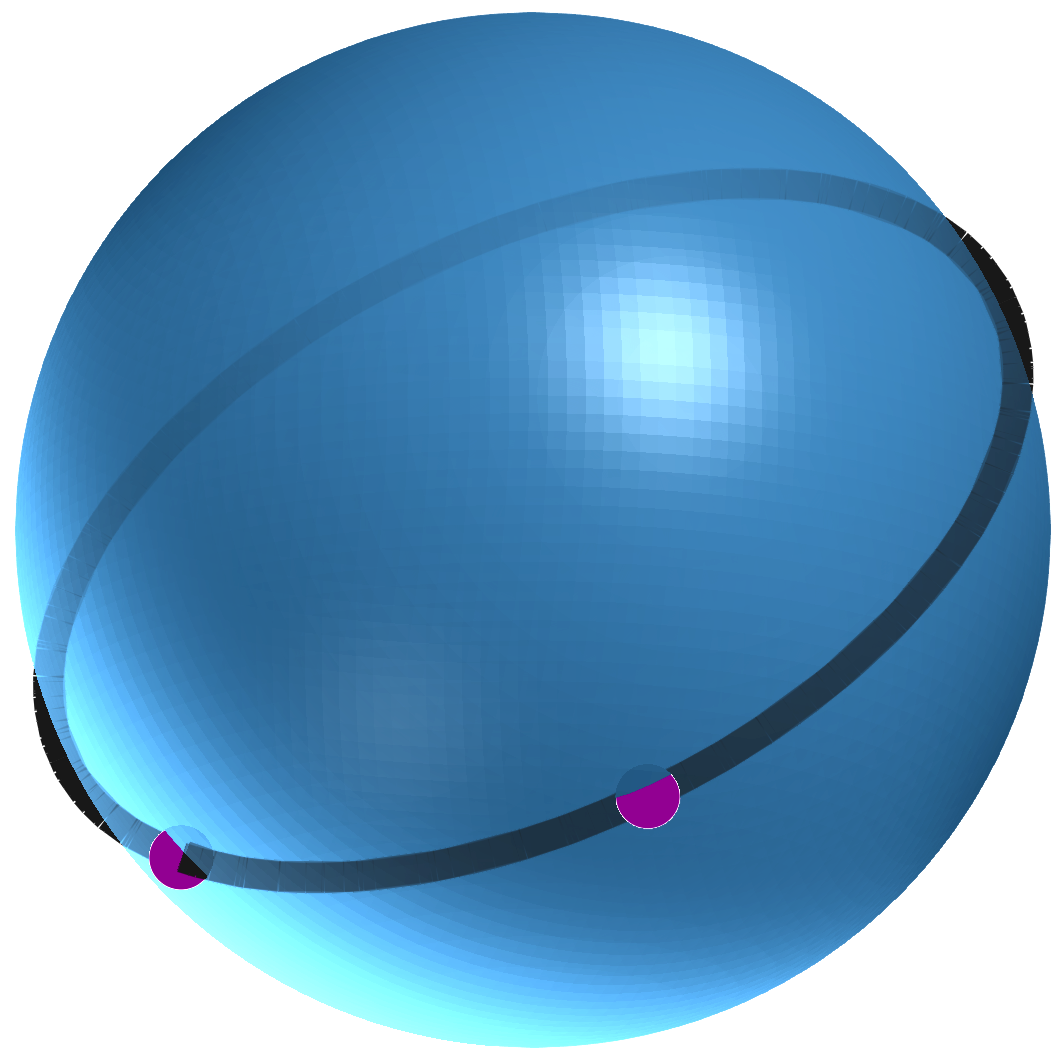}
	\caption{Two different geodesic arcs on the sphere connecting two distinct points.}
	\label{fig:greatcircle}
\end{figure}

For that reason, it is only natural to generalize the theory from this article even further. We could consider multigraphs $G=(V,E)$ with a placement $p:V\rightarrow \M$ so that the multiedges $e\in E$ are given by triples $(i,j,v_{ij})$. In the rigidity theory literature, such graphs are called \struc{gain graphs}, when the edge label $v_{ij}$ describes a certain group action \cite{dewar2026rigiditycompactsurfaceshyperbolic}. 

Here, $i,j\in V$ and $v_{ij}\in T_{p_i}\M$ is the initial velocity of $\gamma_{ij}(t)=\Exp_{p_i}(t\cdot v_{ij})$ so that the geodesic arc $\gamma_{ij}:[0,1]\rightarrow \M$ connects $p_i$ and $p_j$ via $\gamma_{ij}(0)=p_i$ and $\gamma_{ij}(1)=p_j$ with the additional property $\gamma'_{ij}(0)=v_{ij}$. In that case, the geodesic edge length is given by the curve length $L_g(\gamma_{ij})$ instead of the Riemannian distance $\ell_g(p_i,p_j)$. Since geodesics have constant speed, it holds that $L_g(\gamma_{ij})=\|v_{ij}\|_{p_i}$. We could thus more accurately call the graph $G$ a ``geodesically-labeled multigraph''.

Crucially, this construction does not require $\gamma_{ij}$ to be a minimizing geodesic. We can thus weaken the definition of an admissible placement to the invertibility of the differential $d(\Exp_{p_i})_{v_{ij}}$. Then, the Inverse Function Theorem locally guarantees the invertibility of $(x,v)\mapsto (x,\Exp_x(v))$, which provides a locally smooth \struc{branch logarithm} $\mathrm{Log}_{ij}(x,y)$ for the edge $(i,j,v_{ij})$ with the property that $\Exp_x(\mathrm{Log}_{ij}(x,y))=y$ and $\mathrm{Log}_{ij}(p_i,p_j)=v_{ij}$. Contrary to the Riemannian logarithm, the branch logarithm may select a non-minimizing preimage of the exponential map. This lets us define the components of the constraint map as

\[(F_G(q))_{(i,j,v_{ij})} ~=~ \frac{1}{2}\cdot \| \mathrm{Log}_{ij}(q_i,q_j) \|_{q_i}^2.\]

The first-order variation formula for geodesics does not require the geodesics to be minimizing, implying that the rigidity operator has the form

\[\left(d(F_G)_p[u]\right)_{(i,j,v_{ij})}\,=\,-\langle v_{ij},\,u_i\rangle_{p_i} -\langle v_{ji},\,u_j\rangle_{p_j}\]

for $v_{ji}:h=-\gamma'_{ij}(1)$ and $u\in \bigoplus_{i=1}^n T_{p_i}\M$. Notice that 
this expression has the same form as Equation \eqref{eq:row-formula}, only that the Riemannian logarithm $\mathrm{Log}_{p_i}(p_j)$ has been replaced by the branch logarithm $v_{ij}=\mathrm{Log}_{ij}(p_i,p_j)$. Similarly, the Riemannian Hessian can be defined in this setting and conceivably, the local, infinitesimal and second-order constructions from Sections \ref{section:inf_theory}--\ref{section:second-order-theory} can be generalized for the branch logarithm. 

The main disadvantage of this generalization is that the underlying graph needs to take into account the underlying geometry, which means that the rigidity characterization cannot be purely combinatorial and depends on the particular realization. For now, this generalization remains the object of future research.

\section*{Acknowledgements}
I want to thank Klara Stokes for helpful discussions, suggestions and for posing a related problem that motivated me to write this article. I also want to thank Martin Winter for helpful discussions and comments and for his continued interest in this problem. 

\section*{AI Disclosure Statement}
During the preparation of this manuscript, the author used ChatGPT 6 Astra (OpenAI) to assist with drafting parts of the Sections \ref{section:inf_theory}--\ref{section:second-order-theory}. All AI-assisted material was critically reviewed and revised by the author, who takes full responsibility for the content of the manuscript.

\bibliographystyle{alphaurl}
\bibliography{references}

\appendix
	\section{Smooth Manifolds}\label{section:smooth-manifolds}
	We start by introducing a basic definition about the tangent space of smooth manifolds (for a general reference, see e.g.\ \cite{Lee-manifolds}). 
	
	\begin{definition}[Tangent bundle]
		\label{def:tangent-space-and-local-frame}
		Let $\M$ be a smooth manifold. We denote open intervals in $\mathbb{R}$ that contain 0 by $\mathcal{I}_0$. For every point $p\in \mathcal{M}$, consider the set of smooth curves passing through $p$ at $t=0$:
		\[C_p = \{c~\big\lvert ~ c:\mathcal{I}_0\rightarrow \M \text{ is smooth and } c(0)=p\}.\]
		Let $(U,\varphi)$ be a chart of $\M$ around $p$ and consider $c_1,c_2\in C_p$. We define an equivalence relation:
		\[c_1\sim c_2 ~~\Longleftrightarrow ~~ (\varphi \circ c_1)'(0)=(\varphi\circ c_2)'(0).\]
		A \struc{tangent vector at $p$} is an equivalence class $[c]_\sim$ of curves under the equivalence relation $\sim$. 
		The set of all tangent vectors at $p$ is called the \struc{tangent space} $T_p\mathcal{M}=C_p~/\sim$ and is an $\mathbb{R}$-vector space. Its dimension is called the \struc{dimension} $\dim \M$ of $\M$. The disjoint union of all tangent spaces on $\mathcal{M}$ is called the \struc{tangent bundle} $T\mathcal{M} = \bigsqcup_{p\in\mathcal{M}} T_p\mathcal{M}$.
	\end{definition}
	
	\begin{definition}[Local frame]\label{def:local_frame}
		Denote the set of \struc{smooth vector fields} on a manifold $\M$ by $\mathfrak{X}(\M)$:
		\[\mathfrak{X}(\M)~=~\{E:\M\rightarrow T\M~\big\lvert~E \text{ is smooth and } E\lvert_p \,\in T_p\M \text{ for every } p\in\M \}.\]
		A \struc{local frame} for $T\mathcal{M}$ is an ordered collection of smooth vector fields $(E_1,\dots, E_n)$ defined on an open set $U\subset \mathcal{M}$ whose values at each $p\in U$ constitute a basis for $T_p\mathcal{M}$. Given $p\in U$ and a chart $(U,\varphi)$ , we define the \struc{coordinate vector fields}
		\[\partial_i\big\lvert_x ~=~ [t\mapsto \varphi^{-1}(\varphi(x)+t\cdot e_i)]_\sim\]
		for $i=1,\dots, n$ and the standard unit basis $(e_1,\dots, e_n)$ of $\mathbb{R}^n$. We also call these vector fields the \struc{local frame induced by the chart}.
	\end{definition}
	
	Due to the gluing property of charts on a smooth manifold, the coordinate vector fields are smooth. In \Cref{fig:tangent_space}, a visualization of the tangent space to a manifold together with a local frame is provided. It can be thought of as the linearization of the manifold in a given point and we will later see the relation between tangent vectors and velocities of curves on manifolds. 
	
	\begin{figure}[h!]
		\centering
		\includegraphics[width=0.425\linewidth]{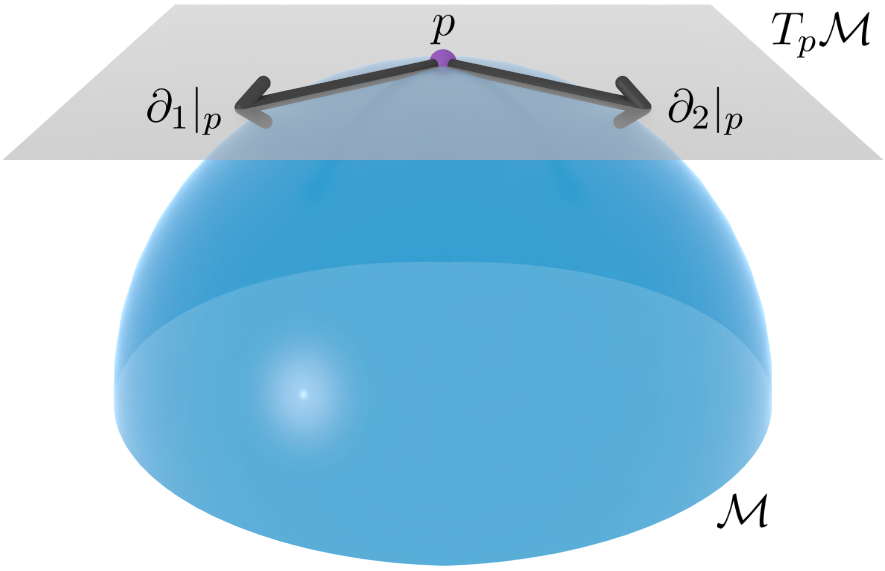}
		\caption{Representation of the tangent space $T_p\mathcal{M}$ at a point $p$ on a manifold $\mathcal{M}$ of dimension 2 together with the induced local frame $(\partial_1,\partial_2)$.}
		\label{fig:tangent_space}
	\end{figure}
	
	Often, a notion of differentiability is useful and we use the concept of a differential in later sections to develop an intrinsic rigidity theory. 
	
	\begin{definition}[Differential]\label{def:differential}
		The \struc{differential} of a smooth function $f:\mathcal{M}\rightarrow \mathbb{R}$ at $p\in \M$ is the linear map $df_p:T_p\M\rightarrow \R$ defined by $df_p(v)=v(f)$. Equivalently, if $v\in T_p\M$ is represented by a smooth curve $c:(-\varepsilon,\varepsilon)\rightarrow \M$ with $c(0)=p$ and $c'(0)=v$ (cf.\ Proposition \ref{prop:tangent_space}), then
		\[df_p(v)=\frac{\partial}{\partial t}\bigg\lvert_{t=0} f(c(t)).\]
	\end{definition}

\section{Riemannian Manifolds}\label{app:riemannian manifolds}
This section introduces several notions from Riemannian geometry. The basic idea of this area of mathematics is to measure local properties on manifolds well-known from Euclidean geometry, such as angles, distances, areas, and curvature. For the necessary differential-geometric background, refer to Appendix \Cref{section:smooth-manifolds}.

\begin{definition}[Riemannian manifold]
	\label{def:riemannian-manifold}
	Let $\mathcal{M}$ be a smooth manifold. Define an inner product $g_p : T_p\mathcal{M}\times T_p \mathcal{M}\rightarrow \mathbb{R}$ for each $p\in \mathcal{M}$. Assume that $g$ is smoothly varying, meaning that for any smooth coordinate chart $(U,x)$ of $\mathcal{M}$ the $(\dim{\mathcal{M}})^2$ functions $g_x( \partial _i\big\lvert_x,\,\partial_j\big\lvert_x): U \rightarrow \mathbb{R}$ are smooth for the induced local frame $(\partial_1,\dots, \partial_n)$. Then, we call $g$ a \struc{Riemannian metric} and the pair $(\mathcal{M}, g)$ is called a \struc{Riemannian manifold}.
\end{definition}
In fact, every smooth manifold admits a Riemannian metric \cite[Prop.\ 2.4]{Lee2018}. Often, we use the notation $\langle \text{-},\text{-}\rangle_p$ for the the Riemannian metric $g_p$ on the tangent space $T_p\mathcal{M}$ and $\|v\|_p$ for the corresponding Riemannian norm given by $\sqrt{\langle v,v\rangle_p}$ with $v\in T_p\M$. An important class of examples for Riemannian manifolds comes from submanifolds of Euclidean space, as they automatically become Riemannian submanifolds by inheriting the metric of the ambient manifold through restriction to the tangent space of the submanifold.

\begin{definition}[Submanifolds]
	\label{def:riemannian-submanifold}
	Suppose $(\mathcal{M},g)$ is a Riemannian manifold and let $\iota: \mathcal{S} \hookrightarrow \mathcal{M}$ be an embedded submanifold. The \struc{pullback metric} $\iota^\star g$ makes $(\mathcal{S},\iota^\star g)$ a Riemannian submanifold. Here, the pullback metric is pointwise defined as the function
	\[(\iota^\star g)_p(X,Y)\, = \,g_p(d\iota_p(X),\,d\iota_p(Y))\]
	for $X,Y\in T_p\mathcal{S}$ with the differential $d\iota_p$ at $p$ mapping 
	$T_pS\rightarrow T_p\mathcal{M}$. 
\end{definition}

Having introduced a differentiable structure on Riemannian (sub-)manifolds, it only seems reasonable to provide analogues to directional derivatives. If the Riemannian metric is the same at every point -- as it is the case in Euclidean space -- we can simply calculate the ordinary gradient and project it onto any submanifold's corresponding tangent space. For general Riemannian metrics, this is not possible, since the local frame generally varies between any two points (cf. Definition \ref{def:tangent-space-and-local-frame}). To make up for this variation, the derivative on Riemannian manifolds needs to carry additional information. This is addressed by defining \struc{connections} $\nabla$ and \struc{covariant derivatives} $\nabla_UV$ along vector fields $U,V\in\mathfrak{X}(\mathcal{M})$ (see \cite[Def. 5.1]{boumal2020intromanifolds}), and by proving that for any Riemannian manifold there exists a unique connection with additional, desirable properties called the \struc{Levi-Civita connection} (cf.\ \cite[Thm. 5.6]{boumal2020intromanifolds}). We can use this theory to define covariant derivatives along curves.

The next step is differentiating vector fields $V\in \mathfrak{X}(c)$ along a curve $c:I\rightarrow \mathcal{M}$, which is a difficult task, because at points $c(t_0)$ and $c(t_1)$ on the curve, the vectors $V(t_0)$ and $V(t_1)$ live in different tangent spaces, so that the difference quotient usually used for defining the derivative does not make sense intrinsically. For that, we define the \struc{covariant derivative along $c$}, which is the unique operator $D_t : \mathfrak{X}(c)\rightarrow \mathfrak{X}(c)$ defined by any connection $\nabla$ which is $\mathbb{R}$-linear and satisfies the Leibniz and Chain rule {\cite[Thm. 5.29]{boumal2020intromanifolds}}. For the Levi-Civita connection $\nabla$, $D_t$ additionally satisfies the product rule. From now on, we will only consider Riemannian manifolds together with their Levi-Civita connection.

A particularly important vector field is given by the tangent vector field of $c$. For that, we define the \struc{velocity} $c'(t)$ at $t$ as the curve's tangent field, i.e. (cf.\ Definition \ref{def:tangent-space-and-local-frame}) \[c'(t)=[s\mapsto c(t+s)]_\sim\in T_{c(t)}\M.\]
The velocity $c'(t)$ consequently is the tangent vector of $c$ at $t$ and we denote the corresponding tangent field by $c'\in \mathfrak{X}(c)$. In the local frame induced by a chart $(U,\varphi)$ that comes with the componentwise maps $\varphi=(\varphi^{(1)},\dots,\varphi^{(n)}):U\rightarrow \mathbb{R}^n$, the tangent vector field is represented by
\[c'(t)=\sum_{i=1}^n\left(\varphi^{(i)}\circ c\right)'(t)\cdot \partial_i\big\lvert_{c(t)}.\]
Staying in line with the physical terminology, we also introduce the acceleration of a curve, which gives rise to a special class of curves.

\begin{definition}[Geodesics]
	\label{def:geodesics}
	Let $c:\mathcal{I}\rightarrow \mathcal{M}$ denote a smooth curve defined on an open interval $\mathcal{I}$. Its \struc{velocity} is the vector field $c'\in \mathfrak{X}(c)$. Its \struc{acceleration} is the smooth vector field $c''\in \mathfrak{X}(c)$ defined by $c''=D_t\,c'$. If $\mathcal{M}$ is a Riemannian manifold with covariant derivative $D_t$ induced by the Levi-Civita connection, a \struc{geodesic} is a smooth curve such that $c''(t)=0$ for all $t\in \mathcal{I}$. A geodesic is called \struc{maximal} if it cannot be extended to a larger interval. A \struc{geodesic arc} is given by $\gamma([a,b])\subset \M$ for a closed interval $[a,b]\subset \mathcal{I}$.
\end{definition} 

\begin{figure}[h!]
	\includegraphics[width=0.35\linewidth]{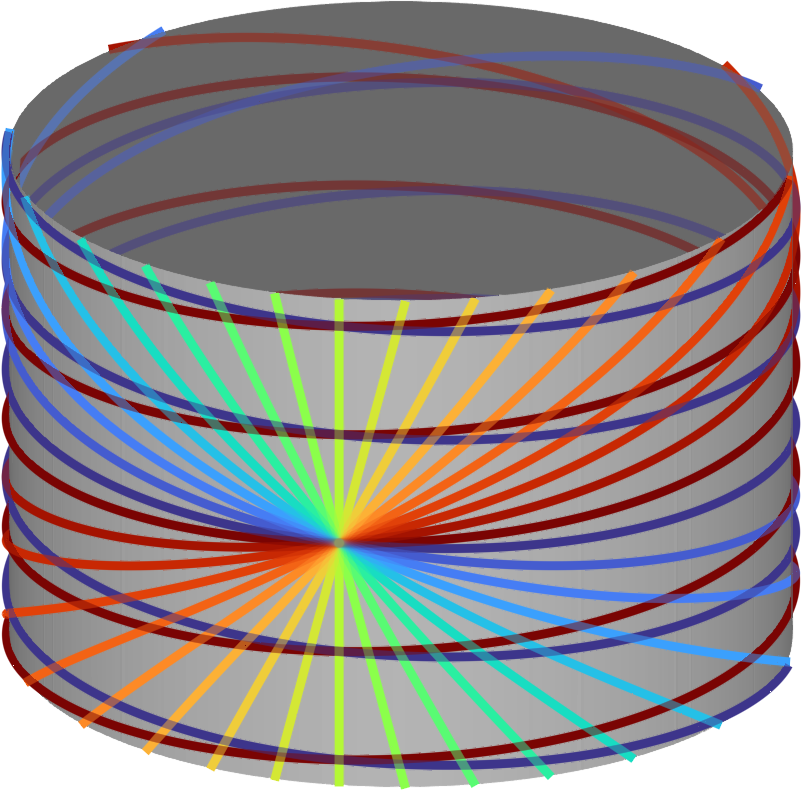}\hspace*{20mm}
	\includegraphics[width=0.35\linewidth]{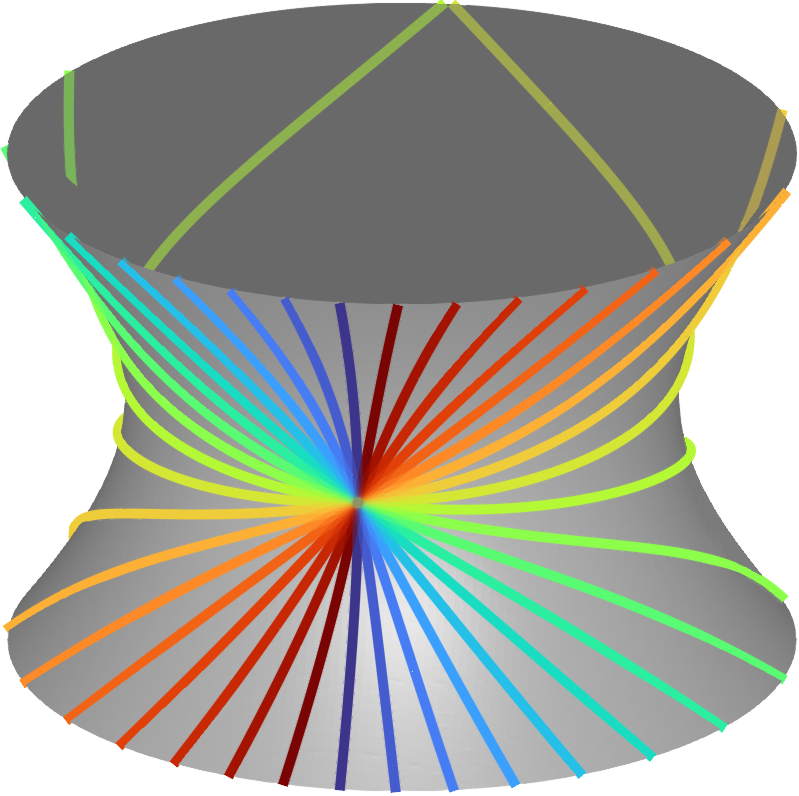}
	\caption{Geodesics through a point on a \textbf{(l.)} cylinder and the \textbf{(r.)} one-sheeted hyperboloid, both equipped with the induced Euclidean metric.}
	\label{fig:geodesics}
\end{figure}

In \Cref{fig:geodesics}, we depict two examples of Riemannian manifolds together with the corresponding geodesics through a fixed point. The cylinder and the one-sheeted hyperboloid are equipped with the Euclidean pullback metric from $\mathbb{R}^3$ (see also \cite[p.\ 70f.]{dissertationhimmelmann}).

Formally, the language of geodesics is not necessary to define the tangent bundle on a manifold (see Definition \ref{def:tangent-space-and-local-frame}). However, it provides a convenient tool to characterize the tangent space, as can be seen in the following.

\begin{proposition}[Characterization of the tangent space]\label{prop:tangent_space}
	Fix a Riemannian manifold $(\mathcal{M},g)$ with Levi-Civita connection $\nabla$. Then, it holds that
	\[T_p\mathcal{M}=\{	c'(0)~\big\lvert ~ c:(-\varepsilon_c, \varepsilon_c)\rightarrow \mathcal{M}\text{ is a geodesic with } c(0)=p \text{ for some } \varepsilon_c>0\}.\] 
	\begin{proof}
		Fix $p\in \mathcal{M}$ and any $v\in T_p\mathcal{M}$. According to  {\cite[Cor.\ 4.28]{Lee2018}} there exists a unique maximal geodesic $\gamma:\mathcal{I}\rightarrow \mathcal{M}$ defined on an open interval $\mathcal{I}\subset \mathbb{R}$ with $0\in \mathcal{I}$ such that $\gamma(0)=p$ and $\gamma'(0)=v$. Fix $\varepsilon>0$ so that $(-\varepsilon,\varepsilon)\subset \mathcal{I}$. By {\cite[Thm.\ 4.27]{Lee2018}}, any geodesic $\Tilde{\gamma}:(-\varepsilon, \varepsilon)\rightarrow \mathcal{M}$ with $\Tilde{\gamma}(0)=p$ and $\Tilde{\gamma}'(0)=v$ agrees with the restriction of $\gamma$ to $(-\varepsilon,\varepsilon)$. Therefore, the tangent space is contained in the set of geodesic velocities. 
		
		For the other inclusion, we first recall {\cite[Thm.\ 4.27]{Lee2018}}, stating that any two geodesics $\gamma, \Tilde{\gamma}$ that pass through $p$ at $t=0$ and have the same velocity at $t=0$ agree on a sufficiently small interval that contains $0$. Hence, the operation $\sim$ is well-defined on the set of geodesic velocities. Since any geodesic is a also smooth curve, the right-hand-side is therefore contained in $T_p\M$. 
	\end{proof}
\end{proposition}

Let us now use the Riemannian metric to define the distance between points on $\mathcal{M}$.

\begin{definition}[Riemannian distance]\label{def:riemannian-distance}
	Let $(\mathcal{M},g)$ be a Riemannian manifold. A smooth curve $\gamma:[a,b]\rightarrow \mathcal{M}$ is \struc{admissible} if there exists a partition of $[a,b]$ into finite sequence of points $a=a_0<a_1<\dots <a_k=b$ so that $\gamma'(t)\neq 0$ for all $a_i<t<a_{i+1}$ and $0\leq i \leq k-1$. For an admissible curve $\gamma$, we define the length of $\gamma$ as 
	\[L_g(\gamma)=\int_{a}^b \|\gamma'(t)\|_{\gamma(t)}\,dt.\]
	The \struc{Riemannian distance} between $p,q\in \mathcal{M}$ is then $\ell_g(p,q)=\inf_\gamma L_g(\gamma)$ where the infimum ranges over all admissible curves.
\end{definition}

It can be shown that every length-minimizing curve is a geodesic \cite[Thm.\ 6.4]{Lee2018}, and conversely, any geodesic locally minimizes the length functional restricted to admissible curves on $\M$ \cite[Thm.\ 6.15]{Lee2018}. However, not all geodesic arcs are minimizers of $L_g(\gamma)$ (see, e.g., \Cref{section:generalization}).

Maximal geodesics can also be used to move on the manifold $\mathcal{M}$, as we show in the following.

\begin{definition}[Exponential map]
	\label{def:exponential-map}
	First, we construct the set \[\mathcal{E}=\{(p,v)\in T\mathcal{M}\,:\,\text{The maximal geodesic }\gamma_v \text{ is defined on an interval containing }[0,1]\}.\] The \struc{exponential map} $\mathrm{Exp}:\mathcal{E}\rightarrow \mathcal{M}$ is then defined by $\mathrm{Exp}(p,v) = \mathrm{Exp}_p(v)=\gamma_v(1)$. The restriction $\mathrm{Exp}_p$ is defined on $\mathcal{E}_p =  \left(\{p\}\times T_p\mathcal{M}\right) \cap \mathcal{E}$. 
\end{definition}

The exponential map is smooth on its domain $\mathcal{E}$, which is an open set in the tangent bundle $T\M$ (cf.\ \cite[Prop.\ 10.17]{boumal2020intromanifolds}) and all tangent space origins are contained in $\mathcal{E}$:
\[\{(x,0)\in T\mathcal{M}~\lvert~x\in \M\}\subset \mathcal{E}.\]
For a point $x\in \M$ and a vector $v\in T_p\M$ with sufficiently short length, the exponential map produces a new point $y=\Exp_p(v) \in \M$. Conversely, given $x,y\in \M$ one may wonder whether it possible to recover the vector $v$. For that, we introduce the following.

\begin{definition}[Injectivity radius {\cite[Def.\ 10.19]{boumal2020intromanifolds}}]\label{def:injectivity-radius}
	The \struc{injectivity radius} $\mathrm{inj}(p)$ of a Riemannian manifold $(\M,g)$ at a point $p\in\M$ is the supremum over radii $r>0$ such that $\mathrm{Exp}_x$ is defined and a diffeomorphism on the open ball $B_x(r)=\{v\in T_x\M~\lvert ~ \|v\|_x<r\}.$
\end{definition}

By the Inverse Function Theorem, it holds that $\mathrm{inj}(x)>0$ \cite[p.\ 255]{boumal2020intromanifolds} and moreover, that $\mathrm{inj}:\M\rightarrow (0,\infty)$ is a continuous function \cite[Prop.\ 10.24]{boumal2020intromanifolds}. For the open ball $U_x=B_x(\mathrm{inj}(x))$ in $T_x\M$, the image $\mathrm{Exp}_x(U_x)$ is a neighborhood of $x$ in $\M$, on which $\mathrm{Exp}_x$ is a diffeomorphism. In particular, this map possesses a smooth inverse function and $v:={\mathrm{Exp}_x}^{-1}(y)$ is the unique shortest tangent vector at $x$ such that $\mathrm{Exp}_x(v)=y$ \cite[p.\ 255]{boumal2020intromanifolds}. This motivates the definition of the Riemannian logarithm at $p\in \mathcal{M}$, which produces an element of the tangent space $T_p\mathcal{M}$.

\begin{definition}[Riemannian logarithm {\cite[Def.\ 10.20]{boumal2020intromanifolds}}]\label{def:riemannian-logarithm}
	For $p\in\M$, the \struc{Riemannian logarithm} is given by
	\[\mathrm{Log}_p(q)\,=\,\mathrm{argmin}_{v\in B_p(\mathrm{inj}(p))} \|v\|_p ~\text{ s.t. }~\mathrm{Exp}_p(v)=q\]
	whose domain is chosen as $\Exp_p(B_p(\mathrm{inj}(p)))\subset \M$. 
\end{definition}

If $\|v\|_x<\mathrm{inj}(x)$, the geodesic $\gamma(t)=\mathrm{Exp}_x(tv)$ on $[0,1]$ is the minimizing curve connecting $x$ to $y=\mathrm{Exp}_x(v)$, which is unique up to reparametrization, with the additional property that $\ell_g(x,y)=\|v\|_x$ and $\mathrm{Log}_x(y)=v$ \cite[Prop.\ 10.22]{boumal2020intromanifolds}.
In particular, let us consider the open set 
\[\mathcal{ID}=\{(x,v)\in T\M~ | ~ \|v\|_x<\mathrm{inj}(x)\}.\]
As previously argued, the exponential map is a diffeomorphism from $\mathcal{ID}$ to 
\begin{eqnarray}\label{eq:U_log}
	\mathcal{U}_{\mathrm{Log}}=\{(x,y)\in \M\times\M ~|~\ell_g(x,y)<\mathrm{inj}(x)\}
\end{eqnarray}
and its inverse $(x,y)\mapsto (x,\mathrm{Log}_x(y))$ is a smooth map from $\mathcal{U}_{\mathrm{Log}}$ to $\mathcal{ID}$ \cite[Cor.\ 10.25]{boumal2020intromanifolds}.

Another useful concept is given by the parallel transport, which yields a way to transfer a tangent vector along a smooth curve to a different tangent space. 

\begin{definition}[Parallel Transport {\cite[Def.\ 10.33--10.35]{boumal2020intromanifolds}}]\label{def:parallel-transport}
	The \struc{parallel transport} along a smooth curve $c:[t_0,t_1]\rightarrow\mathcal{M}$, moving the tangent vectors at $c(t_0)$ to the tangent space at $c(t_1)$ along $c$ is a linear map $\mathcal{T}^c_{t_1\leftarrow t_0}:T_{c(t_0)}\mathcal{M} \rightarrow T_{c(t_1)}\mathcal{M}$ defined by the $\mathcal{T}^c_{t_1\leftarrow t_0}(u)=Z(t_1)$, where $Z\in \mathfrak{X}(c)$ denotes the unique \struc{parallel} vector field such that $Z(t_0)=u$. In other words, it holds that $D_tZ=0$.
\end{definition}

More specifically, $\mathcal{T}$ is the identity when $c$ is constant and preserves inner products between tangent vectors (cf.\ \cite{loclipschitzlinesearch}). 

A particularly important vector field defined along geodesics is given by the following.

\begin{definition}[Jacobi field]\label{def:jacobi-field}
	Let $(\M,g)$ be a Riemannian manifold, let $\gamma$ be a geodesic in $\M$ and let $J$ be a vector field along $\gamma$. For the Levi-Civita connection $\nabla$, we define the \struc{Riemannian curvature tensor} via the equation
	\[R(X,Y)Z=\nabla_X\nabla_Y Z-\nabla_Y\nabla_X Z-\nabla_{[X,Y]}Z\]
	for vector fields $X,Y,Z\in \mathfrak{X}(\M)$ and the \struc{Lie bracket} $[X,Y]=XY-YX$. Then, $J$ is called a \struc{Jacobi field} if it satisfies the \struc{Jacobi equation}
	\[D_t(D_t J)+R(J,\gamma')\gamma'=0.\]
\end{definition}
For a geodesic $\gamma:\mathcal{I}\rightarrow \M$ with $t_0\in \mathcal{I}$, $p=\gamma(t_0)$ every pair of vectors $v,w\in T_p\M$ there exists a unique Jacobi field $J$ (cf.\ \cite[Prop.\ 10.2]{Lee2018}) which satisfies 
\[J(t_0)=v\quad\text{ and }\quad D_t J(t_0)=w.\]
Moreover, if $0\in \mathcal{I}$ and $\gamma(0)=p$, then for every $w\in T_p\M$, the Jacobi field $J$ such that $J(0)=0$ and $D_tJ(0)=w$ is given by
\[J(t)=d(\Exp_p)_{tv}(tw),\]
where $v=\gamma'(0)$ \cite[Prop.\ 10.10]{Lee2018}. 

We also define the concept of a gradient and of a Hessian on Riemannian manifolds (cf.\ {\cite[Def.\ 8.38 and 8.65]{boumal2020intromanifolds}}). In fact, Jacobi fields provide a useful tool for studying the Riemannian Hessian. 

\begin{definition}[Riemannian gradient and Hessian]\label{def:hessian}
	Let $f:\mathcal{M}\rightarrow \mathbb{R}$ be a smooth function on a Riemannian manifold $(\mathcal{M},\langle \text{-},\text{-}\rangle)$. The \struc{Riemannian gradient} $\mathrm{grad}\, f$ is the unique vector field on $\mathcal{M}$ so that 
	\[(f\circ c)'(0)=\langle v,\,\mathrm{grad}\, f\rangle_x\]
	for any smooth curve $c:(-\varepsilon,\varepsilon)\rightarrow \M$ with $c'(0)=v$ and all $(x,v)\in T\M$.
	
	The \struc{Riemannian Hessian}  is the bilinear form $\mathrm{Hess}\,f(x):T_x\mathcal{M}\times T_x\M\rightarrow \mathbb{R}$ defined by 
	\[\mathrm{Hess}\,f(x)[u,v]=\big\langle\nabla_u\mathrm{grad}\,f(x),\,v\rangle_{x}\]
	where $\nabla$ is the Levi-Civita connection on $\M$.
\end{definition}

The Riemannian gradient and Hessian, for instance, find application in Riemannian optimization, as they provide an analogue for the Euclidean gradient and Hessian and can thus be used in local optimization schemes on Riemannian manifolds (cf.\ \cite[§4 and §6]{boumal2020intromanifolds}).

As the last theoretical concept from differential geometry, we define isometries, which will enable us to define a notion of congruence for geometric objects on Riemannian manifolds.

\begin{definition}[Isometries]\label{def:isometries}
	Let $(\M,g)$ and $(\widetilde{\M},\widetilde{g})$ be Riemannian manifolds. An \struc{isometry from $(\M,g)$ to $(\widetilde{\M},\widetilde{g})$} is a diffeomorphism $\varphi: \M\rightarrow \widetilde{\M}$ such that the corresponding pullback metric satisfies $\varphi^*\widetilde{g}=g$. The \struc{set of isometries of $\mathcal{M}$} is then defined as \[\textsc{Isom}(\mathcal{M},g)=\{\varphi:\mathcal{M}\rightarrow\mathcal{M}~\lvert~\varphi \text{ diffeomorphism and } \varphi^*g=g \}.\]
\end{definition}

The set of isometries forms a Lie group and acts properly on $\M$ (cf.\ \cite[Thm.\ 9]{MyersSteenrod1939}). The corresponding \struc{Lie algebra} is denoted by $\mathfrak{isom}(\M,g)$ and contains the infinitesimal isometries of $\M$. Formally, it is given by the tangent space at the identity element $\mathrm{id}$ of the Lie group: $\mathfrak{isom}(\M,g)=T_{\mathrm{id}}\textsc{Isom}(\M,g)$. The dimension of the Lie algebra is at most $\frac{1}{2}n(n+1)$, where $n=\dim\M$ (cf.\ \cite[Thm.\ 3.3]{kobayashinomizu}).

Every element $A\in \mathfrak{isom}(\M,g)$ determines a unique one-parameter subgroup parametrized by 
\[\Phi^A:\mathbb{R}\rightarrow \textsc{Isom}(\M,g),\quad t\mapsto \Phi^A_t\] satisfying $\Phi_0^A=\mathrm{id}$ and $\Phi^A_{s+t}=\Phi^A_{s}\circ\Phi^A_{t}$ for all $s,t\in \mathbb{R}$, namely the exponential map \mbox{$\Phi^A_t=\mathrm{Exp}_{\textsc{Isom}(\M,g)}(tA)$} of the Lie group $\textsc{Isom}(\M,g)$ (cf.\ \cite[Prop.\ 20.8]{Lee-manifolds}). This map has the property that $\frac{\partial}{\partial t}\Phi^A_t\big\lvert_{t=0}=A$.

\section{Explicit prestress stability computation for the regular cube}\label{section:regcube_computations}
Let $S^2\subset\R^3$ be the unit sphere, with cube vertices
\begin{align*}
	p_1&=( 1, 1, 1)/\sqrt3,
	& p_2&=(-1, 1, 1)/\sqrt3,\\
	p_3&=( 1,-1, 1)/\sqrt3,
	& p_4&=(-1,-1, 1)/\sqrt3,\\
	p_5&=( 1, 1,-1)/\sqrt3,
	& p_6&=(-1, 1,-1)/\sqrt3,\\
	p_7&=( 1,-1,-1)/\sqrt3,
	& p_8&=(-1,-1,-1)/\sqrt3.
\end{align*}
Order the twelve edges by
\[
E=(12,34,56,78,13,24,57,68,15,26,37,48).
\]
For $p_i=(a_i,b_i,c_i)/\sqrt3$, with $a_i,b_i,c_i\in\{-1,1\}$,
choose the \emph{orthogonal but not normalized} tangent basis
\[
e_i^{(1)}=(b_i,-a_i,0),
\qquad
e_i^{(2)}=(a_ic_i,b_ic_i,-2).
\]
Write each tangent velocity as
\[
u_i=x_i e_i^{(1)}+y_i e_i^{(2)},\qquad
\mathbf{u}=(x_1,y_1,x_2,y_2,\ldots,x_8,y_8)^\top.
\]

\subsection{Rigidity matrix}\label{section:cube:rig_matrix}
All cube edges have spherical length
\[
\theta=\arccos\left(\frac13\right),
\qquad
\sin\theta=\frac{2\sqrt2}{3}.
\]
For an edge $ij$, the intrinsic first-variation formula gives
\[
(R_{\mathbb{S}^2}(G,p)u)_{ij}
=-\frac{\theta}{\sin\theta}
\bigl(p_j\cdot u_i+p_i\cdot u_j\bigr).
\]
Consequently, in the specified tangent-coordinate and edge ordering,
\[
R_{\mathbb{S}^2}(G,p)=-\sqrt{\frac32}\,\theta\,B,
\]
where the following integer matrix is displayed at a reduced size
to fit the page:
\[
\resizebox{\textwidth}{!}{$
	\displaystyle
	B=
	\left(
	\begin{array}{rrrrrrrrrrrrrrrr}
		-1&-1& 1&-1& 0& 0& 0& 0& 0& 0& 0& 0& 0& 0& 0& 0\\
		0& 0& 0& 0& 1&-1&-1&-1& 0& 0& 0& 0& 0& 0& 0& 0\\
		0& 0& 0& 0& 0& 0& 0& 0&-1& 1& 1& 1& 0& 0& 0& 0\\
		0& 0& 0& 0& 0& 0& 0& 0& 0& 0& 0& 0& 1& 1&-1& 1\\
		1&-1& 0& 0&-1&-1& 0& 0& 0& 0& 0& 0& 0& 0& 0& 0\\
		0& 0&-1&-1& 0& 0& 1&-1& 0& 0& 0& 0& 0& 0& 0& 0\\
		0& 0& 0& 0& 0& 0& 0& 0& 1& 1& 0& 0&-1& 1& 0& 0\\
		0& 0& 0& 0& 0& 0& 0& 0& 0& 0&-1& 1& 0& 0& 1& 1\\
		0& 2& 0& 0& 0& 0& 0& 0& 0&-2& 0& 0& 0& 0& 0& 0\\
		0& 0& 0& 2& 0& 0& 0& 0& 0& 0& 0&-2& 0& 0& 0& 0\\
		0& 0& 0& 0& 0& 2& 0& 0& 0& 0& 0& 0& 0&-2& 0& 0\\
		0& 0& 0& 0& 0& 0& 0& 2& 0& 0& 0& 0& 0& 0& 0&-2
	\end{array}
	\right)
	$}
\]
The column order is $(x_1,y_1,\ldots,x_8,y_8)$, and the rows follow
the displayed edge order. A symbolic row reduction yields
\[
\operatorname{rank}B=11,
\qquad
\dim\ker B=16-11=5.
\]
\subsection{Infinitesimal flexes and self-stresses}\label{section:cube_inf_flexes}
A basis of the kernel is given by the columns of the $16\times5$ matrix
\[
K=\left(
\begin{array}{rrrrr}
	1& 1&1& 1& 0\\
	-1& 1&0& 0& 1\\
	-1& 1&1& 1& 0\\
	-1&-1&0& 0&-1\\
	1&-1&1& 1& 0\\
	1& 1&0& 0&-1\\
	-1&-1&1& 1& 0\\
	1&-1&0& 0& 1\\
	-1&-1&1&-1& 0\\
	-1& 1&0& 0& 1\\
	1&-1&1&-1& 0\\
	-1&-1&0& 0&-1\\
	-1& 1&1&-1& 0\\
	1& 1&0& 0&-1\\
	1& 1&1&-1& 0\\
	1&-1&0& 0& 1
\end{array}
\right)
=\begin{pmatrix} k_x&k_y&k_z&f_1&f_2\end{pmatrix}.
\]
Thus,
\[
BK=0,\qquad \operatorname{rank}K=5,\qquad
\ker R_{\mathbb{S}^2}(G,p)=\operatorname{im}K.
\]
A compact description of these columns is
\begin{align*}
	k_x &: (x_i,y_i)=(a_i c_i,-b_i),\\
	k_y &: (x_i,y_i)=(b_i c_i,a_i),\\
	k_z &: (x_i,y_i)=(1,0),\\
	f_1 &: (x_i,y_i)=(c_i,0),\\
	f_2 &: (x_i,y_i)=(0,a_i b_i).
\end{align*}
The first three vectors are infinitesimal ambient rotations
(up to scaling). Therefore the trivial infinitesimal flex space is given by
\[
\mathcal{K}_p=\operatorname{span}\{k_x,k_y,k_z\},
\]
and a complementary nontrivial flex space is
\[
\ker R_{S^2}(G,p)
=\mathcal{K}_p\oplus\operatorname{span}\{f_1,f_2\}.
\]
The (left) nullspace is one-dimensional:
\[
	\ker R_{\mathbb{S}^2}(G,p)^*
	=\operatorname{span}\{\omega\},
	\qquad
	\omega=
	(\begin{array}{*{12}{c}}
		1&1&1&1&1&1&1&1&1&1&1&1
	\end{array})^{\!\top}.
\]
Equivalently, $B^\top\omega=0$.
Although the tangent bases are not orthonormal, the self-stress
condition is still the ordinary left-nullspace condition,
because the metric on tangent-coordinate space is nondegenerate.
In summary,
\[
	\operatorname{rank}B=11,\quad
	\dim\ker B=5,\quad
	\dim\ker B^\top=1.
\]

\subsection{The stress Hessian}\label{section:cube:edge_hessian}
	The spherical edge Hessians $H_{ij}$ for $ij\in E$ are derived in Equation \eqref{eq:spherical_edge_hessian}. This expression can be further simplified for our example by observing that $\langle p_i,p_j\rangle=\frac{1}{3}$ for the cube's edges $ij$ so that $\frac{\arccos(\langle p_i,\, p_j\rangle)}{\sin(\arccos(\langle p_i,\, p_j\rangle))}=\arccos\left(\frac{1}{3}\right)\cdot\frac{3}{2\sqrt{2}}$. Consequently, for a general nontrivial flex $u=\lambda_1\cdot u^{(1)}+\lambda_2\cdot u^{(2)}$ for $u^{(1)}=\frac{f_1-f_2}{6}$ and $u^{(2)}=\frac{f_1+f_2}{6}$ and the self-stress $\omega$, the stress Hessian $Q_\omega$ has the form
	\begin{eqnarray*}
		\frac{2\sqrt{2}}{\arccos\left(\frac{1}{3}\right)\cdot 3} \cdot Q_\omega(u,u)&=&\frac{2\sqrt{2}}{\arccos\left(\frac{1}{3}\right)\cdot 3}\cdot \sum_{ij\in E}\omega_{ij}\cdot H_{ij}(u,u)\\ 
		&=& \frac{2\sqrt{2}}{\arccos\left(\frac{1}{3}\right)\cdot 3}\cdot \sum_{ij\in E} H_{ij}(u,u)\\
		&=&\sum_{ij\in E} \left(\frac{1}{3}\left( \|u_i\|^2+\|u_j\|^2\right)-2\langle u_i, u_j\rangle \right)\\
		&=& \sum_{i=1}^n\|u_i\|^2-2\sum_{ij\in E} \langle u_i,u_j\rangle~=~(*)\\
	\end{eqnarray*}
	Now, we insert the parametric expression for $u$ to obtain 
	\begin{eqnarray*}
		(*)&=& \sum_{i=1}^n\|\lambda_1 u_i^{(1)}+\lambda_2 u_i^{(2)}\|^2 -  2\sum_{ij\in E} \langle \lambda_1 u_i^{(1)}+\lambda_2 u_i^{(2)},~\lambda_1 u_j^{(1)}+\lambda_2 u_j^{(2)}\rangle  \\
		&=&  \frac{16}{9}\left({\lambda_1}^2-\lambda_1\lambda_2+{\lambda_2}^2\right)+{2\cdot \frac{8}{9}} \left({\lambda_1}^2-\lambda_1\lambda_2+{\lambda_2}^2\right)\\
		&=& \frac{32}{9} \cdot \left[\left(\lambda_1-\frac{\lambda_2}{2}\right)^2+\frac{3}{4}{\lambda_2}^2\right].
	\end{eqnarray*}
\end{document}